\documentclass[11pt]{article}
\usepackage{bbm}
\usepackage{eqnarray,amsmath,amsfonts,amsthm,mathrsfs}
\usepackage{color}
\usepackage{bm}
\usepackage{enumerate}
\usepackage{amssymb}
\usepackage{graphicx}
\usepackage{epsfig}
\usepackage{epsf}
\usepackage{float}
\usepackage{subfigure}
\usepackage{amsfonts,amsmath,amsthm,amssymb,graphicx,float,fancyhdr,multirow,hyperref}
\usepackage{booktabs,longtable,authblk}
\usepackage{mathrsfs,hhline}

\usepackage{fancyhdr}
\usepackage[top=2.5cm, bottom=2.5cm, left=3cm, right=3cm]{geometry}
\usepackage{graphicx}
\newtheorem{theorem}{Theorem}

\newtheorem{corollary}{Corollary}
\newtheorem{definition}{Definition}

\newtheorem{lemma}{Lemma}
\newtheorem{proposition}{Proposition}
\newtheorem{remark}{Remark}

\usepackage[ruled,linesnumbered]{algorithm2e}

\allowdisplaybreaks[4]

\numberwithin{equation}{section}

\title{Expressive Power of Deep Networks on Manifolds: Simultaneous Approximation$^\dag$\footnotetext{\dag~The work of Lei Shi is partially supported by the National Natural Science Foundation of China (Grant No.12171039). The corresponding author is Lei Shi.}}

\author[1]{Hanfei Zhou}
\author[1,2]{Lei Shi}

\affil[1]{School of Mathematical Sciences, \linebreak
Fudan University, Shanghai, 200433, China 
}
\affil[2]{
Shanghai Key Laboratory for Contemporary Applied Mathematics, \linebreak
Fudan University, Shanghai, 200433, China \linebreak
Email:zhouhf23@m.fudan.edu.cn, leishi@fudan.edu.cn
}
\date{}

\begin{document}
	\maketitle
 
\begin{abstract}
A key challenge in scientific machine learning is solving partial differential equations (PDEs) on complex domains, where the curved geometry complicates the approximation of functions and their derivatives required by differential operators. This paper establishes the first simultaneous approximation theory for deep neural networks on $d$-dimensional manifolds $\mathcal{M}^d$. We prove that a constant-depth $\mathrm{ReLU}^{k-1}$ network with bounded weights—a property that plays a crucial role in controlling generalization error—can approximate any function in the Sobolev space $\mathcal{W}_p^{k}(\mathcal{M}^d)$ to an error of $\varepsilon$ in the $\mathcal{W}_p^{s}(\mathcal{M}^d)$ norm, for $k\geq3$ and $s<k$, using $\mathcal{O}(\varepsilon^{-d/(k-s)})$ nonzero parameters. This rate overcomes the curse of dimensionality by depending only on the intrinsic dimension $d$. These results readily extend to functions in Hölder–Zygmund spaces. We complement this result with a matching lower bound, proving our construction is nearly optimal by showing that the required number of parameters matches up to a logarithmic factor. Our proof of the lower bound introduces novel estimates for the Vapnik–Chervonenkis dimension and pseudo-dimension of the network's high-order derivative classes. These complexity bounds provide a theoretical cornerstone for learning PDEs on manifolds involving derivatives. Our analysis reveals that the network architecture leverages a sparse structure to efficiently exploit the manifold's low-dimensional geometry. Finally, we corroborate our theoretical findings with numerical experiments.
\end{abstract}


\section{Introduction}
The remarkable success of deep learning in industrial applications, such as computer vision~\cite{krizhevsky2012imagenet}, natural language processing~\cite{graves2013speech}, and robotics~\cite{gu2017deep}, has spurred its growing adoption in scientific computing. While classical numerical methods—like finite element and spectral methods—possess well-established theoretical foundations for low-dimensional Partial Differential Equations (PDEs), they face two major obstacles: the curse of dimensionality in high-dimensional settings, and the significant computational challenges posed by complex geometries. To overcome these critical challenges, deep neural networks (DNNs) have emerged as powerful surrogate models by leveraging their strong nonlinear approximation capabilities. This has engendered novel numerical paradigms for tackling PDE problems, which can be broadly classified into two categories: unsupervised, physics-driven algorithms, such as 
physics-informed neural networks (PINNs)~\cite{raissi2019physics,karniadakis2021physics,cai2021physics, zhou2025weak}, and supervised, data-driven algorithms like neural operators~\cite{li2020fourier, kovachki2023neural}. For a comprehensive review of deep learning approaches for solving PDEs, we refer the reader to~\cite{de2024numerical}. These advances highlight the advantage of deep learning in high-dimensional, nonlinear problems and open new avenues for modeling complex physical phenomena beyond the reach of traditional methods.

This empirical success is theoretically grounded in the universal approximation theorems of the 1980s and 1990s~\cite{cybenko1989approximation, pinkus1999approximation, hornik1991approximation}. These seminal works established that even a simple, shallow network can approximate any continuous function to arbitrary precision. However, these classical theorems are purely qualitative. They guarantee the existence of a suitable network approximator (i.e., they establish class density), but the results are non-constructive: they provide no quantitative guidance on how network architecture—such as its width or depth—should scale with a prescribed accuracy level $\varepsilon$.

The practical success of deep learning has prompted a renewed focus on approximation theory, guiding the field toward establishing quantitative error bounds. A landmark paper~\cite{yarotsky2017error} demonstrated that, by exploiting the compositional structure of deep networks, a $\mathrm{ReLU}$ network with $\mathcal{O}(\varepsilon^{-d/n})$ parameters can achieve an approximation error of $\varepsilon$ for any Sobolev function $f \in W_\infty^n([0,1]^d)$. This seminal work motivated extensive research extending quantitative analyses to various contexts, including different activation functions~\cite{de2021approximation, jiao2023deep}, smoothness assumptions~\cite{lu2021deep}, approximation norms~\cite{guhring2020error}, and network architectures~\cite{zhou2020universality}. Notably, subsequent work has also demonstrated that neural networks can simultaneously approximate a function and its derivatives, showcasing their capacity to capture higher-order information~\cite{guhring2020error,belomestny2023simultaneous,yang2023nearly}. This is precisely the capability that underpins their success in solving PDEs. Since the training objective, such as the PINNs loss, consists of high-order derivative terms, ensuring the residual converges to zero fundamentally requires the network to possess this simultaneous approximation capacity. However, existing results in Euclidean settings tie error rates to the ambient dimension, without yet fully exploiting the low-dimensional geometric structures ubiquitous in real-world applications.

Although most deep learning-based PDE solvers operate in Euclidean settings, a multitude of critical problems are naturally formulated on manifolds. These applications—ranging from geophysical modeling~\cite{haltiner1980numerical, bonito2020divergence, bachini2021intrinsic} and brain modeling~\cite{brosch2013manifold} to computer graphics~\cite{Turk1991GeneratingTO, bertalmio2001navier}—provide the central motivation for developing manifold-aware solvers. The prevalence of such problems has also been a primary driver for the growth of adjacent fields, notably manifold learning~\cite{izenman2012introduction}.

Concurrently, theoretical work has begun to establish approximation guarantees for H{\"o}lder functions on manifolds. Leveraging tools like coordinate charts or random projections (i.e., the Johnson-Lindenstrauss lemma~\cite{eftekhari2015new}), these studies achieve approximation rates dependent solely on the manifold's intrinsic dimension $d$~\cite{SchmidtHieber2019DeepRN, Chen2019NonparametricRO, Labate2024LowDA}. Despite this progress, these results are insufficient for solving PDEs due to two key limitations. First, the error metrics are restricted to $L_p$ norms, which provide no control over derivatives—a prerequisite for handling operators such as the Laplace-Beltrami operator $\Delta_{\mathcal{M}}$. Second, the H{\"o}lder spaces in use often depend on an ambient Euclidean metric~\cite{SchmidtHieber2019DeepRN,Labate2024LowDA}, making their definition non-intrinsic and complicating practical verification.

In this work, we establish the approximation rates achieved by deep neural networks with respect to high-order Sobolev norms for functions defined on a broad class of $d$-dimensional Riemannian manifolds $\mathcal{M}^d$ with bounded geometry. Our approach marks a significant departure from prior work, which has primarily focused on highly specialized settings. A notable example is a pioneering study on simultaneous approximation for the sphere~\cite{lei2025solving}. This work established rates for convolutional neural networks (CNNs) by ingeniously exploiting the sphere's unique zonal structure. On the sphere, the reproducing kernels of diffusion polynomial spaces depend only on the geodesic distance, which in turn is a simple function of the ambient Euclidean inner product. This specific geometric property was key to leveraging the known efficiency of CNNs in modeling functions of inner products~\cite{zhou2020universality,zhou2020theory}, thereby achieving simultaneous approximation. Our theory, in contrast, relies on no such specialized structural assumptions.

The CNN-based methodology is not limited to a specific sphere. Its underlying mechanism—the zonal structure—is a fundamental characteristic of two-point homogeneous spaces~\cite{wang1952two}, a class distinguished by the existence of a convenient addition formula~\cite{evarist1975addition}. This class of manifolds, however, is exceptionally restrictive. It has been fully classified and consists of just a few canonical families: spheres, projective manifolds, and the Cayley projective plane. Our work takes a more general approach, free from such constraints, with sparse localization constructions that may in turn illuminate the feasibility of CNN-like architectures beyond these special cases.

To establish the optimality of these rates, we derive a matching lower bound for the Sobolev space on manifolds $\mathcal{W}^k_\infty(\mathcal{M}^d)$. Specifically, we demonstrate that approximating any function in this space to an accuracy $\varepsilon$ (in the $\mathcal{W}^s_\infty(\mathcal{M}^d)$ norm) necessitates a constant-depth $\mathrm{ReLU}^{k-1}$ network with at least $S$ parameters, where $ S\,\log S\;\gtrsim\;\varepsilon^{-\frac{d}{\,k - s\,}}.$ This result confirms that the network complexity is essentially governed by the manifold's intrinsic dimension $d$, not the ambient one. Our proof introduces novel estimates for both the Vapnik-Chervonenkis (VC) dimension and the pseudo-dimension of the function class realized by the $s$-th order derivatives of $\mathrm{ReLU}^{k-1}$ networks. These complexity bounds represent a contribution of independent interest. Whereas existing complexity analyses for $\mathrm{ReLU}$ networks have focused on the network's output or its first derivatives~\cite{bartlett2019nearly,yang2023nearly}, our work provides the first such bounds for high-order derivative classes. This is a critical prerequisite for developing a complete convergence theory for PINNs governed by high-order PDEs, which our results address.

This paper is organized as follows. \autoref{Section: Preliminaries} introduces our assumptions, several equivalent definitions of Sobolev norms on manifolds, and the neural network architecture. \autoref{Section: Main Result} presents our main theorems, establishing matching upper and lower bounds for approximation rates in Sobolev and Hölder–Zygmund spaces, which demonstrates the optimality of the network complexity. In~\autoref{Experiments} we give numerical experiments that validate the theoretical findings and provide architectural design insights. We discuss the implications and future research directions in \autoref{Section: Discussion}. Proofs of our main results are provided in \autoref{Section: Proofs}, while all notation, auxiliary lemmas, and their proofs are collected in the \hyperref[allapp]{Appendix}.

\section{Preliminaries}\label{Section: Preliminaries}
In this section, we state our main assumption. We consider a compact, connected, complete, $d$-dimensional Riemannian manifold $\mathcal{M}^d$ with bounded geometry (see~\autoref{Appendix: Notations} for relevant background). A key consequence of the bounded geometry assumption is the existence of a uniform family of local charts, which are essential for defining Sobolev spaces on manifolds. This is formalized in the following result.
\begin{proposition}[Theorem 1 in~\cite{de2021reproducing}]\label{proposition: local charts}
Let $\mathcal{M}^d$ satisfy the aforementioned assumptions. Then for any radius $r > 0$ that is sufficiently small (i.e., less than the injectivity radius of $\mathcal{M}^d$), there exists a finite smooth atlas $\{U_j, \psi_j\}_{j=1}^K$ and a corresponding smooth partition of unity $\{\rho_j\}_{j=1}^K$ with the following properties:
\begin{enumerate}[(i)]
    \item The atlas: The collection of charts is constructed from a finite set of points $\{m_j\}_{j=1}^K \subset \mathcal{M}^d$ such that $\{U_j\}_{j=1}^K$ forms a finite open cover of $\mathcal{M}^d$. For each $j \in \{1, \dots, L\}$, the chart is given by:
    \begin{equation*} \label{eq:compact_atlas}
        U_j = B(m_j, r), \qquad \psi_j: U_j \to \mathbb{R}^d, \quad \text{where} \quad \psi_j(m) = \exp_{m_j}^{-1}(m).
    \end{equation*}

    \item The partition of unity: The family of smooth functions $\{\rho_j\}_{j=1}^K$ forms a partition of unity corresponding to the cover $\{U_j\}$, satisfying for all $j \in \{1, \dots, L\}$:
    \begin{equation*} 
        0 \le \rho_j \le 1, \qquad \mathrm{supp}(\rho_j) \subset U_j, \qquad \text{and} \qquad \sum_{j=1}^K \rho_j = 1.
    \end{equation*}
\end{enumerate}
  
\end{proposition}
\subsection{Sobolev and H{\"o}lder-Zygmund Functions on Manifolds}\label{Subsection: Sobolev and Holder-Zygmund functions on manifolds}
Sobolev spaces on a manifold are typically defined in two ways. The local approach defines norms on coordinate charts, which are then pieced together into a global definition using a partition of unity. Alternatively, the global approach constructs a coordinate-independent space directly from the manifold's intrinsic geometry. For integer orders, this global construction can be based on the volume form and Levi-Civita connection, which provides a framework for embedding inequalities~\cite{hebey1996sobolev}. For non-integer orders, the space is defined via the spectral theory of the Laplace-Beltrami operator. This spectral approach, developed using heat kernel methods~\cite{strichartz1983analysis}, further establishes that for compact manifolds and sufficient smoothness index, these spaces become Reproducing Kernel Hilbert Spaces with a characterizable kernel~\cite{de2021reproducing}.

In this paper, we adopt this framework to derive approximation rates for neural networks targeting Sobolev functions. We assume the manifold satisfies the necessary conditions for these different definitions to be equivalent. We follow the framework of~\cite{de2021reproducing}, which defines and establishes the equivalence of various Sobolev norms for the Sobolev exponent $p=2$. These equivalences extend to $p \geq 1$ with minor adjustments.

\begin{theorem}[Theorem 3 in~\cite{de2021reproducing}]\label{theorem: sobolev equivalence}
    Let $\mathcal{M}^d$ be a $d$-dimensional, compact, connected, and complete Riemannian manifold with bounded geometry. For a smoothness order $s>0$, the following conditions are equivalent:
    \begin{enumerate}[(i)]
        \item 
        \begin{equation}\label{equation: sobolev definition1}
             \|f\|_{\mathcal{W}^{s}_2,1}^2 := \sum_{j} \Bigl\|\;(\rho_j f)\circ \psi_j^{-1}\Bigr\|_{W^{s}_2(\mathbb{R}^d)}^2 < +\infty,
        \end{equation}
        where $\{(U_j,\psi_j)\}$ is a system of local coordinate charts on $\mathcal{M}^d$ and $\{\rho_j\}$ is a corresponding smooth partition of unity given by~\autoref{proposition: local charts}.
        
        \item There exists a function $g \in L^2(\mathcal{M}^d)$ such that:
        \begin{equation}\label{equation: sobolev definition2}
            f = (I+\Delta_{\mathcal{M}})^{-s/2}\, g, \quad  \|f\|_{\mathcal{W}^{s}_2,2} := \|g\|_{L^2(\mathcal{M}^d)}.
        \end{equation}
        Here, the operator $(I+\Delta_{\mathcal{M}})^{-s/2}$ is defined via the spectral calculus with the function $\Phi(\lambda) = (1+\lambda)^{-s/2}$.
        
        \item The function $f$ lies in the domain of the operator $\Delta_{\mathcal{M}}^{s/2}$, i.e., $f \in \mathrm{dom}(\Delta_{\mathcal{M}}^{s/2})$, and satisfies:
        \begin{equation}\label{equation: sobolev definition3}
            \|f\|_{\mathcal{W}^{s}_2,3}^2 := \|f\|_{L^2(\mathcal{M}^d)}^2 + \|\Delta_{\mathcal{M}}^{s/2} f\|_{L^2(\mathcal{M}^d)}^2 < +\infty.
        \end{equation}
    \end{enumerate}
    That is, there exist constants $c_1, c_2, c_3, c_4 > 0$, independent of $f$, such that the norms are equivalent:
    \begin{equation*}
        c_1\|f\|_{\mathcal{W}^{s}_2,1} \leq c_2 \|f\|_{\mathcal{W}^{s}_2,2} \leq c_3 \|f\|_{\mathcal{W}^{s}_2,3} \leq c_4\|f\|_{\mathcal{W}^{s}_2,1}.
    \end{equation*}
    Furthermore, for a positive integer $s\in \mathbb{N}_+$, the Sobolev space $\mathcal{W}^s_2(\mathcal{M}^d)$ can be defined as the closure of smooth functions under the norm involving covariant derivatives:
    \begin{equation}\label{equation: sobolev definition4}
        \mathcal{W}^s_2(\mathcal{M}^d) = \overline{\left\{f\in C^\infty(\mathcal{M}^d)\ \Big|\ \sum_{\ell=0}^s \|\nabla^\ell f\|_{L^2(\mathcal{M}^d)}^2 < +\infty\right\}},
    \end{equation}
    and the associated norm is equivalent to those defined above.
\end{theorem}

For a sketch of the proof, see the proof of Theorem 3 in~\cite{de2021reproducing}. In this work, we primarily adopt the Sobolev norm given by~\eqref{equation: sobolev definition1}. As the theorem demonstrates, this choice of norm is natural, well-justified, and frequently appears in various applications.

Building on this, we now define the H{\"o}lder-Zygmund space on the manifold as:
\begin{equation*}
    \mathcal{CH}^s (\mathcal{M}^d) := \left\{f \ \Big|\ \|f\|_{CH^s(\mathcal{M}^d)} := \sup_{j} \|(\rho_jf)\circ\psi_j^{-1}\|_{CH^s(\mathbb{R}^d)} < +\infty\right\}.
\end{equation*}
Here, $CH^s(\mathbb{R}^d)$ denotes the classical H{\"o}lder-Zygmund space on $\mathbb{R}^d$, which is equivalent to the Besov space $B^s_{\infty,\infty}$ (see~\cite[Section 1.2.2 and 1.5.1]{Triebel1992}). We consider the case where $s$ is not an integer, for which the H{\"o}lder-Zygmund space coincides with the classical H{\"o}lder space. Specifically, its norm is defined as:
\begin{equation}\label{equation: CH definition}
    \|f\|_{\mathcal{CH}^s(\mathbb{R}^d)} := \max\left\{ \sup_{|\alpha|\leq\lfloor s\rfloor }\|D^{\alpha }f\|_{L_\infty(\mathbb{R}^d)} , \sup_{|\alpha|=\lfloor s\rfloor }\sup_{x \not=y} \frac{|D^\alpha f(x)-D^\alpha f(y)|}{\|x-y\|^{s-\lfloor s\rfloor}} \right\}.
\end{equation}
A proof of this equivalence can be found in~\cite[Sec. 1.2.2]{Triebel1992}.
\subsection{Mathematical Definition of Neural Networks}\label{Subsubsection: Mathematical Definition of Neural Networks}

Given $\Omega \subset \mathbb{R}^m$, $\mathbf{b} = (b_1, b_2, \dots, b_r)^T$ and $\mathbf{y} = (y_1, y_2, \dots, y_r)^T$, we define the shifted activation function as $\sigma_{\mathbf{b}}(\mathbf{y}) = (\sigma(y_1-b_1), \sigma(y_2-b_2), \dots, \sigma(y_r-b_r))^T: \mathbb{R}^r \to \mathbb{R}^r$. The neural networks can be expressed as a family of real-valued functions of the form
\begin{equation}\label{equation: Neural Network}
    f:\mathbb{R}^{d}\to\mathbb{R},\quad x\mapsto f(x)= W_{L}\sigma_{\mathbf{v}_L}W_{L-1}\sigma_{\mathbf{v}_{L-1}}\cdots W_1\sigma_{\mathbf{v}_1}W_0x.
\end{equation}
Our neural network architecture consists of an input layer of dimension $m$, $L$ hidden layers, and a scalar output. The dimensions are given by $p_0=d$, $p_{L+1}=1$, and the width of the $i$-th hidden layer is $p_i$. A critical component for our analysis is the activation function. The standard $\mathrm{ReLU}$ function is fundamentally limited for simultaneous approximation, as its vanishing higher-order derivatives prevent it from capturing the structure of functions in Sobolev spaces. To overcome this, we employ the $\mathrm{ReLU}^{k-1}$ activation, defined as $\sigma(x) = \max\{x,0\}^{k-1}$ for an integer $k \ge 3$. This function possesses the necessary smoothness to approximate a target function and its derivatives concurrently. The neural network architecture is formally parameterized by a sequence of weight matrices $\{W_i\}_{i=0}^L$ and bias vectors $\{\mathbf{v}_i\}_{i=1}^L$, which are optimized during training. In this paper, we specifically view neural networks as implementing the parameterized function class $f(\cdot;\{W_k,\mathbf{v}_k\})$ defined in equation~\eqref{equation: Neural Network}.

The neural network function space can be characterized by its depth $L$, width vectors $\{p_i\}_{i=0}^{L+1}$, and the number of non-zero parameters in  weights $\{W_i\}_{i=0}^L$ and biases $\{\mathbf{v}_i\}_{i=1}^L$. In addition, the complexity of this space is also determined by the $\|\cdot\|_{\infty}$-bounds of the network parameters and the $\|\cdot\|_{L_{\infty}}$-bounds of the network output $f$ in form~\eqref{equation: Neural Network}. Let $(Q,G) \in [0,\infty)^2$ and $(S,B,F) \in [0,\infty]^3$. Denote $\max\{m_1, m_2, ..., m_l\}$ by $m_1 \vee m_2  \vee \ldots  \vee m_L$ and $\min\{m_1, m_2, ..., m_l\}$ by $m_1 \wedge m_2  \wedge \ldots  \wedge m_L$. The function space of neural networks is then defined as
 \begin{equation}\label{equation: Assumption for hypothesis space}
    \begin{aligned}
        &\mathcal{F}(Q,G, S, B, F) \\
        ={}& \left\{
        f:\mathbb{R}^d \to\mathbb{R} \left\vert
            \begin{array}{l}
                 \text{$f$ is defined by \eqref{equation: Neural Network}, satisfying that } \\
                 \text{$L \leq Q$ and $p_1 \vee p_2  \vee \ldots  \vee p_L \leq G $, } \\
                 \text{$\left(\sum_{i=0}^L\|W_i\|_{0}\right) +\left(\sum_{i=1}^L \|\mathbf{v}_i\|_0 \right)\leq S $, }\\
                 \text{$ \sup_{k=0,1,\ldots,L} \|W_k\|_\infty \lor \sup_{k=1,\ldots,L} \|\mathbf{v}_k\|_\infty \leq B,$ } \\
                 \text{and $\|f\|_{L_{\infty}} \leq F .$} \\
             \end{array}
             \right.
             \right\}.
         \end{aligned}
    \end{equation}
    
 For brevity, we will often denote the function space from Definition~\eqref{equation: Assumption for hypothesis space} as $\mathcal{F}$. The parameters in this definition may be set to infinity; for instance, $B=\infty$ or $F=\infty$ signifies the absence of constraints on the network's weight magnitudes or its $L_\infty$ norm, while $S=\infty$ corresponds to a dense (non-sparse) architecture. Rigorous mathematical analysis shows that in neural network approximation theorems, such as those presented in~\cite{yarotsky2017error,guhring2020error}, the parameter bound $B$  must diverge to infinity as the approximation error tends to zero. Consequently, the space  $\mathcal{F}(Q,G,S,\infty,F)$ is non‐compact, which in turn gives it an infinite uniform covering number and presents substantial challenges for subsequent statistical analysis. Therefore, a key focus of our construction is the careful characterization of these parameter bounds. Specifically, we demonstrate that the desired approximation can be achieved while maintaining a fixed weight-norm bound of $B=1$.

\section{Main Results}\label{Section: Main Result}
In this section, we present the main theorems of this paper. We first present the upper bound result.

\noindent\textbf{Upper bound.} We begin with a theorem on the simultaneous approximation of Sobolev functions on manifolds:

\begin{theorem}\label{theorem:main result sobolev}
    Let $\mathcal{M}^d \subset \mathbb{R}^D$ be a smooth, compact, connected, and complete $d$-dimensional Riemannian manifold with bounded geometry, where $d \ll D$. Let $p \geq 1$, let $k \geq 3$ be an integer, and let $f \in \mathcal{W}_p^k(\mathcal{M}^d)$ be as defined in~\autoref{theorem: sobolev equivalence}.
    
    Then for any $\epsilon > 0$, there exists a fully connected deep neural network $g \in \mathcal{F}(L,W,S,1,\infty)$ using the $\mathrm{ReLU}^{k-1}$ activation function that simultaneously satisfies:
    \begin{equation*}
        \|f-g\|_{\mathcal{W}_p^s(\mathcal{M}^d)} \leq \epsilon \quad \text{for all } s \in \{0, 1, \ldots, k-1\}.
    \end{equation*}
    Furthermore, the network parameters are bounded by:
    \begin{equation*}
        L \leq C, \quad W \leq C \epsilon^{-\frac{d}{k-s}}, \quad S \leq C \epsilon^{-\frac{d}{k-s}},
    \end{equation*}
    for a constant $C = C(p,d,D,\mathcal{M}^d,k)$. A precise characterization of $C$ is provided in the proof (see~\autoref{Subsection: Upper Bounds} and~\autoref{Appendix: Auxiliary Lemma and Proofs}) but is omitted here for brevity.
\end{theorem}
\begin{remark}
    Notably, the network complexity depends only on the manifold's intrinsic dimension $d$. This demonstrates the ability of neural networks to overcome the curse of dimensionality, as their complexity scales with $d$ rather than the potentially much larger ambient dimension $D$. An additional result, presented as a corollary, extends this principle to the Hölder-Zygmund family of functions, demonstrating that the network's ability to exploit low-dimensional geometric structure also applies to this function class.
\end{remark}

\begin{remark}
    Although $\mathrm{ReLU}^{k-1}$ activations may appear susceptible to gradient explosion, this is circumvented here since the depth remains constant. Comparable approximation rates are attainable with deeper $\mathrm{ReQU}=\mathrm{ReLU}^2$ networks: while $\mathrm{ReLU}^{k-1}$ can represent spline functions in a single layer, it often requires excessive width even for simple functions (see~\autoref{lemma: Decomposition of x^s to ReLU^k}), whereas $\mathrm{ReQU}$ avoids this inefficiency~\cite{belomestny2023simultaneous}. The most economical architecture is therefore a final multi-layer $\mathrm{ReQU}$ composed after an initial $\mathrm{ReLU}^{k-1}$ layer. In contrast, multi-layer $\mathrm{ReLU}$ networks remain unsuitable: once the partition-of-unity structure is encoded by splines, the bit-extraction mechanism~\cite{bartlett2019nearly,lu2021deep,yang2023nearly} is no longer applicable. Moreover, since $\mathrm{ReLU}$ can only perform piecewise-linear interpolation, it leads to unbounded local derivatives; consequently, higher-order activations are required to construct smoother partition-of-unity representations.
\end{remark}

\begin{corollary}\label{corollary: main result Holder}
    Let $\mathcal{M}^d \subset \mathbb{R}^D$ be a smooth, compact, connected, and complete $d$-dimensional Riemannian manifold with bounded geometry, where $d \ll D$. Let $k > 2$ be a non-integer, let $p \geq 1$, and let $f \in \mathcal{CH}^k(\mathcal{M}^d)$ be as defined in~\eqref{equation: CH definition}.
    
    Then for any $\epsilon > 0$, there exists a fully connected deep neural network $g \in \mathcal{F}(L,W,S,1,\infty)$ with the $\mathrm{ReLU}^{\lfloor k \rfloor}$ activation function that simultaneously satisfies:
    \begin{equation*}
        \|f-g\|_{\mathcal{CH}^s(\mathcal{M}^d)} \leq \epsilon \quad \text{for all non-integer} \, s <k.
    \end{equation*}
    Furthermore, the network parameters are bounded by:
    \begin{equation*}
        L \leq C, \quad W \leq C \epsilon^{-\frac{d}{k-s}}, \quad S \leq C \epsilon^{-\frac{d}{k-s}},
    \end{equation*}
    for a constant $C = C(p,d,D,\mathcal{M}^d,k)$.
\end{corollary}
\begin{remark}
    For the case of pure function approximation ($s=0$), our network size bound of $O(\epsilon^{-d/k})$ aligns with previously established results for learning on low-dimensional manifolds~\cite{Chen2019efficient, Chen2019NonparametricRO,SchmidtHieber2019DeepRN}. The main contribution, however, is the simultaneous approximation of smooth derivatives for $s \ge 1$. This demonstrates that a deep neural network is capable of adaptively capturing the intrinsic low-dimensional geometry of the data embedded in a high-dimensional space, including its essential differential information. Crucially, our construction achieves this remarkable adaptability while maintaining uniformly bounded parameters (as signified by $g \in \mathcal{F}(L,W,S,1,\infty)$). This is a significant feature, as many theoretical constructions require network weights to grow unboundedly as the approximation error vanishes~\cite{yarotsky2017error,guhring2020error}. By establishing these approximation rates that depend on the intrinsic dimension $d$, within a well-behaved, compact-like hypothesis space, our work provides a critical step toward developing robust generalization bounds.
\end{remark}

\noindent\textbf{Lower bound.} Having established the upper bounds, we now consider the optimal rates. To this end, we will derive a matching lower bound on the required network complexity. The cornerstone of our lower bound proof is a novel estimate of the complexity of the network's derivative classes, which we measure using the VC dimension and the pseudo-dimension. We begin by recalling the formal definitions (e.g., see~\cite{anthony2009neural}).

\begin{definition}
    Consider a function class $ \mathcal{H} : \mathcal{X} \to \mathbb{R} $ and a set $ A = \{a_i\}_{i=1}^m $, consisting of $m$ points in the input space  $\mathcal{X}$. Let $\mathrm{sgn}(\mathcal{H}) = \{\mathrm{sgn}(h) : h \in \mathcal{H}\} $ denote the set of binary functions $\mathcal{X} \to \{0,1\}$ induced by $\mathcal{H}$. We say that $\mathcal{H}$ or $\mathrm{sgn}(\mathcal{H})$ shatters the set $A$ if $\mathrm{sgn}(\mathcal{H})$ generates all possible dichotomies of $A$, formally expressed as $ \#\{\mathrm{sgn}(h)|_A \in \{0,1\}^m : h \in \mathcal{H}\} = 2^m. $ The VC-dimension of $ \mathcal{H} $ or $ \mathrm{sgn}(\mathcal{H}) $ is defined as the cardinality of the largest subset it can shatter, denoted by $\mathrm{VCDim}(\mathcal{H})$. Furthermore, a set $A$ is said to be pseudo-shattered by $\mathcal{H}$ if there exist thresholds $b_1, b_2, \ldots, b_m \in \mathbb{R}$ such that for every binary vector $v \in \{0,1\}^m$, there exists a function $h_v \in \mathcal{H}$ satisfying $\mathrm{sgn}(h_v(a_i) - b_i) = v_i$ for all $1 \leq i \leq m$. The pseudo-dimension of $\mathcal{H}$, denoted by $\mathrm{PDim}(\mathcal{H})$, is defined as the maximum cardinality of a subset $A \subset \mathcal{X}$ that can be pseudo-shattered by $\mathcal{H}$.
\end{definition}

For any multi-index $\alpha$ with $\lvert\alpha\rvert\le s$, we define the class of its derivative functions as
\begin{equation*}
    D^{\alpha}\mathcal{F} := \{D^{\alpha}u : u \in \mathcal{F}\}.
\end{equation*}
With this notation, we can now state our estimate for the complexity of the class $D^{\alpha}\mathcal{F}$.
\begin{theorem}\label{theorem: VC-dimension}
Let $\mathcal{F}=\mathcal{F}(Q,G,S,\infty,\infty)$ denote the class of functions realized by $\mathrm{ReLU}^{k-1}$ networks, as formally defined in Definition~\eqref{equation: Assumption for hypothesis space}. Then for any multi-index $\alpha$ with $|\alpha| \leq s$, we have
 \begin{equation*}
 \mathrm{PDim}(D^{\alpha}\mathcal{F})\vee \mathrm{VCDim}(D^{\alpha}\mathcal{F}) \lesssim s \cdot 3^s \cdot SQ\log S.
\end{equation*}
The constant here is independent of the network parameters $Q, G,$ and $S$.
\end{theorem}
The proof strategy is inspired by~\cite[Theorem 6]{bartlett2019nearly} and~\cite[Theorem 4.13]{jiao2022rate}. The central step involves reconstructing each function in the derivative class $D^{\alpha}\mathcal{F}$ with a new, surrogate network. This network employs a combination of activation functions, namely $\mathrm{ReLU}^{t}$ for $t=s, s+1, \ldots, k-1$. This key constructive argument is rigorously presented in~\autoref{lemma: reproduce drivatives}. Consequently, the problem reduces to estimating the complexity of this new function class.

We now present our main approximation lower bound.

\begin{theorem}\label{theorem: lower bounds}
Let $\epsilon \in (0,1)$. To achieve a uniform error of $\epsilon$ under the $s$-th order Sobolev norm for all functions $f \in \mathcal{W}^k_\infty(\mathcal{M}^d)$ (as defined in \eqref{equation: sobolev definition1}), any $\mathrm{ReLU}^{k-1}$ network architecture $\mathcal{F}(Q,G,S,\infty,\infty)$ must employ a number of nonzero parameters, $S$, that satisfies
 \begin{equation*}
SQ\log S \gtrsim \epsilon^{-\frac{d}{k-s}}.
\end{equation*}
\end{theorem}

The complete proofs for~\autoref{theorem: VC-dimension} and~\autoref{theorem: lower bounds} are provided in~\autoref{Subsection: Lower Bounds}.
\begin{remark}
Notably, this is the first result to establish a lower bound for the simultaneous approximation rates of neural networks in the context of manifolds. Previous lower bounds for simultaneous approximation were limited to first-order derivatives in the Euclidean setting~\cite{yang2023nearly}. This lower bound demonstrates that the approximation rates achieved by constant-depth $\mathrm{ReLU}^{k-1}$ networks are nearly optimal, as our upper and lower bounds on the required number of nonzero parameters match up to a logarithmic factor. Crucially, the bound depends only on the intrinsic dimension $d$, revealing that approximation on manifolds is governed by the same intrinsic complexity as in $\mathbb{R}^d$. Our lower bound is established in the $L_\infty$ norm, which is often of greater interest in computational settings since it provides pointwise control of the approximation error. While recent works~\cite{achour2022general,siegel2023optimal} have provided lower bounds in $L_p$ norms ($1\leq p < \infty$), their methods are often specific to the $L_p$ setting and Euclidean geometry. A notable example is the proof in~\cite{siegel2023optimal}, which relies on decomposing a cube into smaller, translated copies—a construction that does not generally hold for partitions on manifolds. Consequently, extending our analysis to $L_p$ Sobolev spaces on manifolds remains an important open problem and a promising direction for future research.
\end{remark}

\section{Experiments}\label{Experiments}
In this section, we conduct numerical experiments to validate our theoretical findings and provide architectural design insights. The experiments are implemented using Python scripts powered by the PyTorch framework \href{https://pytorch.org/}{https://pytorch.org/}. The scripts can be downloaded from \href{https://github.com/hanfei27/ReLUkNNonManifolds}{https://github.com/hanfei27/ReLUkNNonManifolds}.
\subsection{Numerical Experiment Setup}

Denote the unit sphere by $\mathbb S^2=\{x\in\mathbb R^3:\|x\|=1\}.$ On \(\mathbb S^2\) the target function is the standardly normalized third-order zonal spherical harmonic
\[
f(x)=Y_{3,1}(x),\qquad x\in\mathbb S^2,
\]
which satisfies the spectral relation \(\Delta_{\mathbb S^2}Y_{3,1}=-\lambda_3 Y_{3,1}\) with \(\lambda_3=12\). Definitions and the relevant spectral properties of the Laplace--Beltrami operator \(\Delta_{\mathbb S^2}\) and of spherical harmonics are collected in \autoref{appendix: LB and hd}.

We generate a weighted sample set \(\{(x_i,w_i)\}_{i=1}^N\) on \(\mathbb S^2\) using a Fibonacci grid~\cite{hannay2004fibonacci} and approximate surface integrals by the discrete quadrature
\[
\int_{\mathbb S^2} g(x)\,\mathrm d\omega(x)\approx\sum_{i=1}^N w_i\,g(x_i).
\]

The model is a fully connected \(\mathrm{ReLU}^k\) network \(u_\theta:\mathbb R^3\to\mathbb R\), evaluated at the sampled points. During training we optimise only the \(L^2\) (function value) discrepancy; explicitly the training objective on the sphere is
\[
\mathcal L_{\mathbb S^2}^{\mathrm{train}}(\theta)
=\sum_{i=1}^N w_i\big|u_\theta(x_i)-f(x_i)\big|^2.
\]
Training uses the Adam optimizer~\cite{kingma2015adam} with initial learning rate \(\eta=10^{-3}\), batch size \(\mathrm{BS}=2048\), and runs for \(T\) iterations. Optimization is performed by stochastic minibatch updates over the weighted samples and the best parameters found during training are saved for evaluation.

For evaluation we measure Sobolev-style errors that probe first- and second-order accuracy, even though these terms are not included in the training objective. Concretely we compute the weighted mean squared errors
\begin{align*}
\mathrm{WMSE}_{f,\mathbb S^2}
&=\sum_{i=1}^N w_i\big|u_\theta(x_i)-f(x_i)\big|^2,\\
\mathrm{WMSE}_{\nabla,\mathbb S^2}
&=\sum_{i=1}^N w_i\big\|\nabla_{\mathbb S^2} u_\theta(x_i)-\nabla_{\mathbb S^2} f(x_i)\big\|^2,\\
\mathrm{WMSE}_{\Delta,\mathbb S^2}
&=\sum_{i=1}^N w_i\big|\Delta_{\mathbb S^2}u_\theta(x_i)-\Delta_{\mathbb S^2}f(x_i)\big|^2.
\end{align*}
The tangential gradients and the Laplace--Beltrami operator used in these test-time metrics are computed numerically as described in \autoref{appendix: LB and hd}. Spectrally, the gradient and Laplace metrics correspond to multiplying spherical-harmonic coefficients by \(\lambda_n\) and \(\lambda_n^2\), respectively, so these weighted MSEs probe approximation quality across frequency bands and provide a Sobolev-style assessment of the learned model.

We repeat the validation on an embedded torus \(\mathbb T^2_{R,r}\subset\mathbb R^3\) parameterized by
\[
X_{R,r}(u,v)=\big((R+r\cos v)\cos u,\; (R+r\cos v)\sin u,\; r\sin v\big),\qquad u,v\in[0,2\pi),
\]
with \(R>r>0\) (in experiments we commonly take \(R=1.5,\; r=0.5\)). The torus target is taken from an analytic Fourier family, for example
\[
g(x)=A\cos(m\,u(x))+B\sin(n\,v(x)),
\]
where \(u(x),v(x)\) are the angular coordinates recovered from \(x\). The Laplace--Beltrami action on this family admits a closed-form expression (see \autoref{appendix: LB and hd}), hence \(\Delta_{\mathbb T^2}g\) is available exactly for testing.

Weighted torus samples \(\{(y_j,\tilde w_j)\}_{j=1}^M\) are drawn by sampling \((u_j,v_j)\) uniformly on \([0,2\pi)^2\), mapping via \(X_{R,r}\), and rescaling weights by the Jacobian \(\tilde w_j\propto r(R+r\cos v_j)\) so that
\[
\int_{\mathbb T^2_{R,r}} h\,\mathrm d\tilde\omega(y)\approx\sum_{j=1}^M \tilde w_j\,h(y_j).
\]
On the torus we train with the \(L^2\) objective
\[
\mathcal L_{\mathbb T^2}^{\mathrm{train}}(\theta)=\sum_{j=1}^M \tilde w_j\big|u_\theta(y_j)-g(y_j)\big|^2,
\]
using the same Adam schedule as on the sphere. Performance is reported via \(\mathrm{WMSE}_{g,\mathbb T^2}\), \(\mathrm{WMSE}_{\nabla,\mathbb T^2}\) and \(\mathrm{WMSE}_{\Delta,\mathbb T^2}\); under the toroidal Fourier basis these correspond spectrally to multiplying coefficients by \(1\), \(\mu_{m,n}\) and \(\mu_{m,n}^2\), respectively, and thus quantify accuracy across the doubly periodic frequency bands consistent with the Sobolev norm on \(\mathbb T^2_{R,r}\).
\subsection{Network Architectural Ablation Study}
We evaluated the effect of activation order by sweeping
\(\mathrm{act\_k\_list}=\{1,2,3,4,5,6,7\}\) (elementwise activation \(\sigma_k(z)=(\mathrm{ReLU}(z))^k\); \(k=1\) is standard ReLU). Experiments were performed on both the unit sphere \(\mathbb S^2\) and the embedded torus \(\mathbb T^2_{R,r}\); all runs used the identical baseline configuration given in the main text. Each \(k\) was repeated 5 times on each manifold and~\autoref{tab:actk_final_loss} below reports the final training loss (mean \(\pm\) std, 5 repeats) for both manifolds. 
\begin{table}[!ht]
\centering
\caption{Effect of activation order \(k\) on final training loss (mean \(\pm\) std, 5 repeats).}
\label{tab:actk_final_loss}
\begin{tabular}{c c c}
\toprule
$k$ & training loss (sphere) & training loss (torus) \\
\midrule
1 & $\mathtt{9.42e-05}\pm\mathtt{2.70e-05}$ & $\mathtt{3.22e-05}\pm\mathtt{4.85e-06}$ \\
2 & $\mathtt{4.19e-05}\pm\mathtt{3.99e-05}$ & $\mathtt{1.48e-05}\pm\mathtt{6.96e-06}$ \\
3 & $\mathtt{5.23e-05}\pm\mathtt{1.85e-05}$ & $\mathtt{1.64e-05}\pm\mathtt{1.52e-05}$ \\
4 & $\mathtt{9.75e-06}\pm\mathtt{2.11e-06}$ & $\mathtt{6.75e-06}\pm\mathtt{4.78e-06}$ \\
5 & $\mathtt{8.45e-06}\pm\mathtt{2.65e-06}$ & $\mathtt{3.35e-06}\pm\mathtt{2.02e-06}$ \\
6 & $\mathtt{4.36e-06}\pm\mathtt{1.42e-06}$ & $\mathtt{4.40e-06}\pm\mathtt{3.75e-06}$ \\
7 & $\mathtt{2.82e-06}\pm\mathtt{7.09e-07}$ & $\mathtt{3.68e-06}\pm\mathtt{1.22e-06}$ \\
\bottomrule
\end{tabular}
\end{table}

Additionally, we plot the function value, gradient, and Laplace--Beltrami loss components as functions of the activation order \(k\). Solid lines indicate the sample mean across five repeats and shaded bands denote \(\pm1\sigma\). Results for both manifolds are shown: the unit sphere \(\mathbb S^2\) and the embedded torus \(\mathbb T^2_{R,r}\).
\begin{figure}[!ht]
    \centering
    \includegraphics[width=\linewidth]{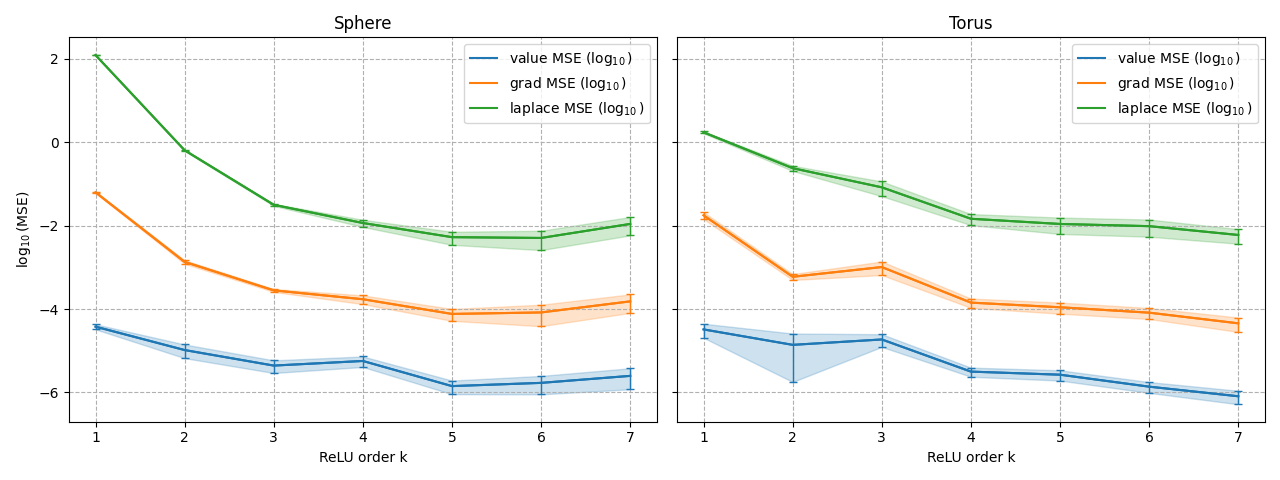}
    \caption{Component-wise mean MSEs (log$_{10}$ scale) versus ReLU activation order $k$ on two manifolds.  Left: unit sphere $\mathbb S^2$.  Right: embedded torus $\mathbb T^2_{R,r}$.  Plotted are function value MSE (blue), tangential gradient MSE (orange) and Laplace--Beltrami MSE (green).  Curves denote the sample mean over five independent runs; shaded bands indicate $\pm1\sigma$ and markers/error bars show the sample standard deviation.}
    \label{fig:actk_components}
\end{figure}
In both the sphere and torus experiments we observe that the component MSEs (function value, tangential gradient and Laplace–Beltrami) decrease as the activation order $k$ increases (plots are shown on a log scale). This behavior aligns with two complementary theoretical points. First, standard $\mathrm{ReLU}$ networks ($k$=1) are consistent with the simultaneous-approximation results of~\cite{guhring2020error}: they can achieve control of the function and its first derivatives, i.e. first-order simultaneous approximation. Second, as we show in our analysis, ReLU networks do not, in general, provide simultaneous approximation beyond first order and therefore cannot guarantee uniform control of second-order quantities such as the Laplace–Beltrami. Increasing the activation order mitigates this limitation because higher-order activations are smoother at the kink and thus enable the network to form higher-degree piecewise-polynomial, spline-like approximants. This additional smoothness enables network exploit higher-order information of the target and assemble higher-degree local polynomial pieces, which improves approximation in Sobolev norms that are sensitive to second derivatives; accordingly, the Laplace–Beltrami MSE decreases as $k$ increases in our experiments. To illustrate this effect,~\autoref{fig:spatial_errors} shows pointwise spatial error fields (function value error, tangential gradient norm error, and absolute Laplace error) for \(k=1\) and \(k=5\) on both manifolds, highlighting the smoothing effect and local reductions in gradient / Laplace errors produced by the higher-order activation.
\begin{figure}[!ht]
  \centering
  \includegraphics[height=9cm,width=0.8\textwidth,keepaspectratio=false]{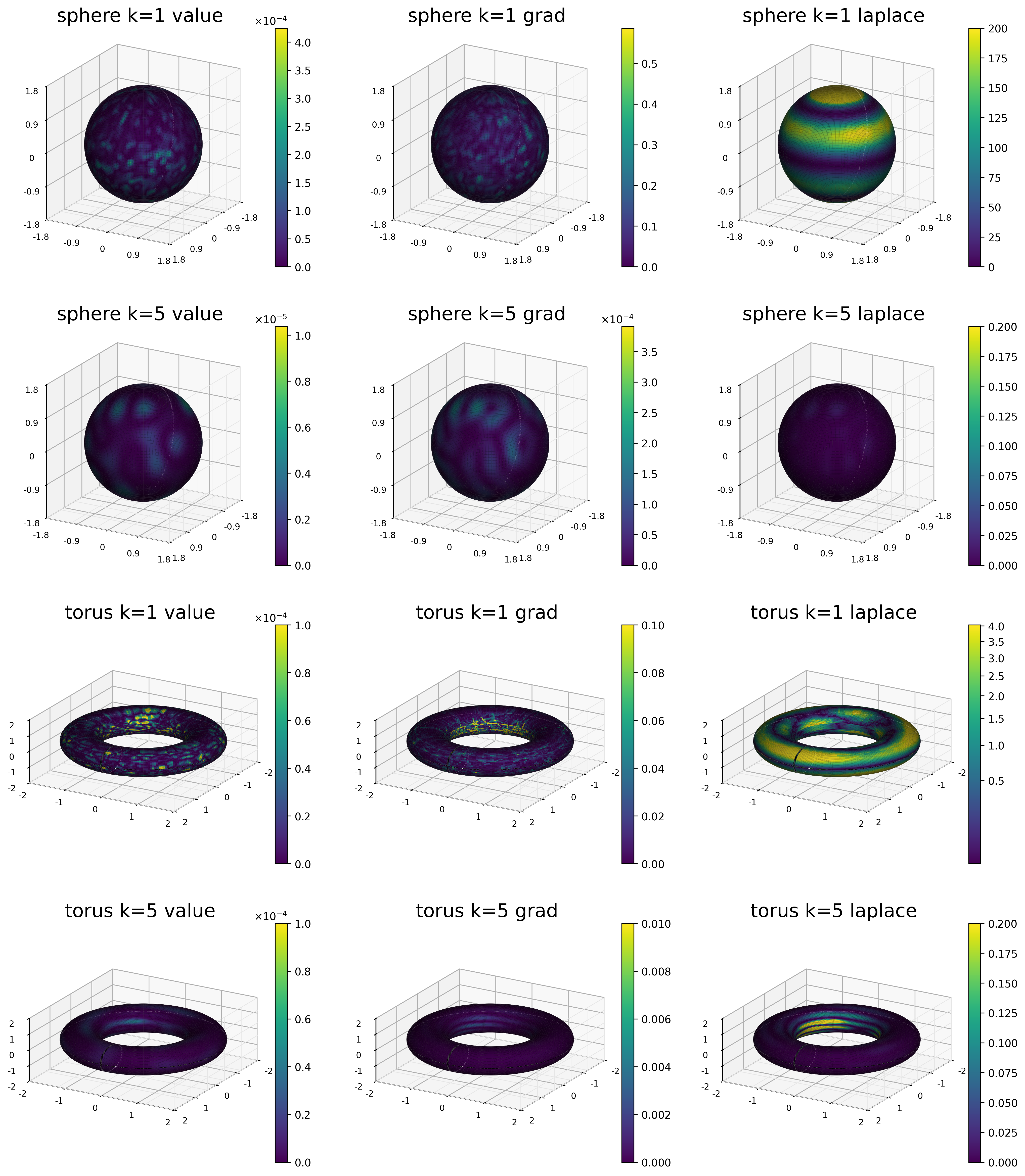}
  \caption{Spatial error maps on the sphere and torus.}
  \label{fig:spatial_errors}
\end{figure}

Having characterized the effect of activation order \(k\) and observed diminishing returns for large \(k\), we henceforth fix \(k=4\) and examine network width \(m\in\{16,32,64,96,128,256\}\) (depth \(L=2\)), holding all other hyperparameters constant. The network depth was fixed at 2 so that width is the only variable. All runs used the same baseline configuration: total training steps \(T=5000\), total samples \(N=20000\), mini-batch size \(\mathrm{BS}=2048\), learning rate \(\eta=10^{-3}\), and random seed fixed for repeatability. Each width was trained independently for 5 repeats; reported statistics are the sample mean and sample standard deviation across repeats. The table below lists the three component MSEs (function value, gradient, Laplace--Beltrami).
\begin{table}[!ht]
\centering
\caption{Final training loss (mean $\pm$ std, 5 repeats) for different network widths.}
\label{tab:width_final_loss}
\begin{tabular}{c c c}
\toprule
width $m$ & training loss (sphere) & training loss (torus) \\
\midrule
16  & $\mathtt{1.07e-03}\pm\mathtt{1.66e-03}$ & $\mathtt{1.01e-03}\pm\mathtt{2.14e-04}$ \\
32  & $\mathtt{5.69e-05}\pm\mathtt{2.47e-05}$ & $\mathtt{2.42e-04}\pm\mathtt{5.26e-05}$ \\
64  & $\mathtt{2.12e-05}\pm\mathtt{1.54e-05}$ & $\mathtt{8.92e-05}\pm\mathtt{2.39e-05}$ \\
96  & $\mathtt{1.80e-05}\pm\mathtt{1.26e-05}$ & $\mathtt{7.09e-05}\pm\mathtt{1.63e-05}$ \\
128 & $\mathtt{5.18e-06}\pm\mathtt{2.42e-06}$ & $\mathtt{5.46e-05}\pm\mathtt{1.22e-05}$ \\
256 & $\mathtt{1.81e-06}\pm\mathtt{3.80e-07}$ & $\mathtt{4.47e-05}\pm\mathtt{1.07e-05}$ \\
\bottomrule
\end{tabular}
\end{table}

We also plot the component-wise MSEs in two log–log diagrams that both display MSE versus network width: one subplot for the sphere and one for the torus (horizontal axis: \(\log_2(\text{width})\)).  Each curve shows the sample mean across five repeats with shaded bands indicating \(\pm1\sigma\).

\begin{figure}[!ht]
    \centering
    \includegraphics[width=\linewidth]{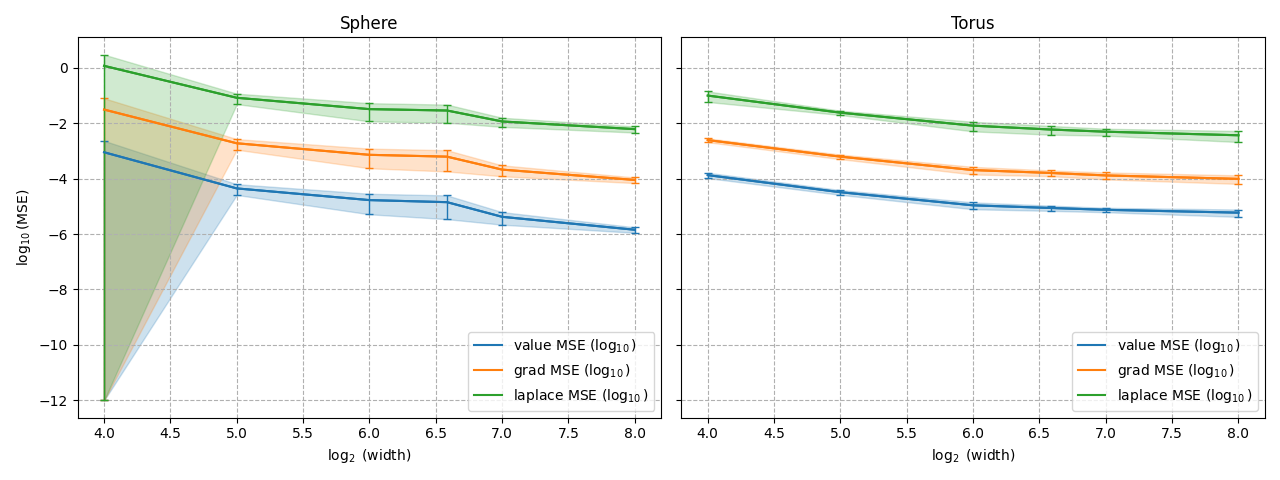}
    \caption{Component-wise mean MSEs versus network width (log--log). Left: unit sphere \(\mathbb S^2\). Right: embedded torus \(\mathbb T^2_{R,r}\). Plotted are function value MSE (blue), tangential gradient MSE (orange) and Laplace--Beltrami MSE (green). Solid lines denote the sample mean over five independent runs; shaded bands indicate \(\pm1\sigma\). The horizontal axis is shown as \(\log_2(\text{width})\) and the vertical axis as \(\log_{10}(\text{MSE})\).}
    \label{fig:width_component_losses}
\end{figure}
Both plots (sphere and torus) display an approximately linear decline on log–log axes, so we fit simple linear models in the logarithmic domain to estimate power-law decay exponents for each error component. The fitted lines indicate algebraic (power-law) decay of the error with model size (width / parameter count) in both cases, which qualitatively supports the mechanism predicted by our theory. Quantitatively, however, the fitted slopes do not match the theoretical worst-case exponents: the theory provides conservative bounds valid in the worst case, whereas our numerical targets are much smoother than worst-case examples — the sphere experiments use a spherical harmonic and the torus experiments use Fourier-type targets, both of which possess strong regularity. Consequently exact agreement of exponents is not expected, although the observed algebraic decay corroborates the theory’s qualitative prediction that larger models yield polynomially improving Sobolev accuracy. The fitted log–log slopes for each error component (fit of $\log_{10}(\mathrm{MSE})$ versus $\log_{10}(\text{parameter count})$) are reported in ~\autoref{tab:slopes_by_manifold}.  The slopes quantify the observed algebraic decay rates; the sphere exhibits substantially steeper decay than the torus, reflecting the higher regularity of the spherical-harmonic target used in those runs.

\begin{table}[!ht]
  \centering
  \caption{Estimated log–log slopes $\alpha$ (fit of $\log_{10}(\mathrm{MSE})$ vs.\ $\log_{10}$ parameter count).}
  \label{tab:slopes_by_manifold}
  \begin{tabular}{l c c}
    \toprule
    Component & Sphere (log–log slope) & Torus (log–log slope) \\
    \midrule
    Value  & $-2.1708$ & $-1.1323$ \\
    Grad   & $-1.9832$ & $-1.1741$ \\
    Laplace & $-1.7896$ & $-1.2026$ \\
    \bottomrule
  \end{tabular}
\end{table}

We also studied network depth. Empirically, for activation order \(k>3\) increasing depth tends to destabilize optimization and can prevent convergence. To mitigate this for high-order \(\mathrm{ReLU}^k\) in deep nets we applied a systematic stabilization protocol: (i) a variance rescaling based on half-normal moments so that
\[
\operatorname{Var}\big(s_k\mathrm{ReLU}^{\,k}(z)\big)=\operatorname{Var}\big(\mathrm{ReLU}(z)\big),
\]
with scalar \(s_k>0\) estimated from the half-normal assumption on pre-activations; (ii) insertion of \texttt{LayerNorm} after each hidden linear layer; (iii) upper bounds on activations and weights (\(\text{activation\_clamp\_max}=3.0,\ \text{weight\_clip}=0.1\)); and (iv) an optimization schedule with reduced initial learning rate \(\eta=5\times10^{-4}\). With these measures a \(\mathrm{ReLU}^k\) network with \(k=5\) and depth \(L=5\) stably converges on the spherical-harmonic target, reducing the Laplace--Beltrami residual to \(\mathcal{O}(10^{-2})\). However, the deep model shows substantially longer training time and no clear error advantage over the shallower \(L=2\) baseline; this empirical finding aligns with our theoretical message that, absent explicit parameter-norm constraints, modest depth (e.g. \(L=3\)) already suffices to obtain the predicted approximation rates.

\subsection{Mixed Architectures Study}
We assessed whether $\mathrm{ReLU}$–$\mathrm{ReLU}^k$ mixtures can combine the optimization advantages of ReLU with the higher-order approximation power of smooth activations. Experimental settings were held fixed: network width $m=128$, depth $L=4$, training samples $N=10{,}000$, training steps $T=4{,}000$, and \texttt{LayerNorm} applied uniformly. The activation pattern was the only variable: a tail-mixed family $[1,1,1,5]$ (first three hidden layers $\mathrm{ReLU}$, final hidden layer $\mathrm{ReLU}^5$) and single high-order placements $[5,1,1,1]$, $[1,5,1,1]$, $[1,1,5,1]$ were tested. To exclude simple optimizer remedies, we ran ablations with activation clamping on and off, as well as with several gradient-clipping thresholds.

Across all configurations and optimization variants, the weighted Laplace error remained on the order of $\sim 4.4\times10^{1}$ to $\sim 1.25\times10^{2}$, comparable to the first-order error level of standard ReLU networks; the corresponding gradient error likewise stayed at $O(10^{-2})$ with no systematic improvement. These results indicate that, under the present experimental conditions, ReLU–ReLU$^k$ mixed architectures (regardless of high-order unit placement) do not induce the expected high-order ($H^2$-type) convergence. In principle, a $\mathrm{ReLU}$–$\mathrm{ReLU}^k$ mixed architecture can realize the smooth approximants required for \(H^2\)-type convergence; however, in practice the optimization process offers no guarantee that the $\mathrm{ReLU}$ components will be driven to the specific smooth configurations needed, so the theoretical feasibility does not translate into a reliable practical route to second-order convergence. Consequently, for improved numerical performance when second-order control is important, we recommend using globally smoother activations (for example, GeLU or softplus) or, as a pragmatic compromise that preserves gradient flow, LeakyReLU; these alternatives typically yield more stable optimization and better empirical reduction of Laplace-sensitive errors than architectures that mix non-smooth ReLU blocks with high-order units. However, a full exploration of these alternative activations and their theoretical properties lies beyond the scope of this paper.

\section{Discussion}\label{Section: Discussion}
Our work establishes the first result on the simultaneous approximation rates of $\mathrm{ReLU}^{k-1}$ networks for Sobolev and Hölder-Zygmund functions on manifolds, measured in higher-order norms. We demonstrate that these rates are nearly optimal with respect to the number of non-zero network parameters. Crucially, they depend only on the manifold's intrinsic dimension $d$, not the ambient dimension, thereby overcoming the curse of dimensionality. Moreover, a natural extension is to consider functions with anisotropic smoothness on manifolds, such as those in Korobov or anisotropic Besov spaces. It is well-established that in Euclidean settings, neural networks can effectively overcome the curse of dimensionality when approximating such functions~\cite{yang2023optimal,suzuki2021deep,suzuki2019adaptivity,yang2024near}. However, a key challenge is to identify concrete instances of functions possessing this anisotropic smoothness on manifolds. While some studies suggest that solutions to certain PDEs may exhibit these properties in Euclidean space~\cite{yserentant2004regularity,griebel2007sparse}, the situation on manifolds remains largely unexplored. Investigating this constitutes a promising direction for future research. 

Our results show that neural networks can exploit intrinsic geometric structure, laying a foundation for solving forward and inverse PDE problems on manifolds. Since upper bounds are typically constructive, matching lower bounds play a distinct role in confirming the sharpness of the rates. Establishing such optimality would highlight neural networks as not only powerful approximators but also reliable tools for PDE analysis and computation. A natural next step is to develop PDE solvers based on these guarantees and rigorously analyze their convergence and optimality. 

For PDE problems, the PINNs framework aligns naturally with our approximation setting, making the results directly applicable to convergence analysis on manifolds. In the presence of boundaries, stronger approximation guarantees are needed to enforce hard constraints from boundary conditions~\cite{lei2025solving,zhou2025weak}, making this an important subdirection for further analysis. More broadly, a complete theory must also account for generalization error, governed by the complexity of the hypothesis class (e.g., its VC-dimension or pseudo-dimension). By establishing new bounds on these complexity measures for derivative classes of $\mathrm{ReLU}^{k-1}$ networks, our work provides a key step toward a comprehensive convergence theory for PINNs on manifolds.

\section{Proof of Main Results}\label{Section: Proofs}
\subsection{Upper Bounds}\label{Subsection: Upper Bounds}
we first provide the proof of upper bound results. The proof is structured as follows: we first outline the key ideas and strategies employed in the proof, followed by a detailed step-by-step construction of the neural networks that achieve the desired approximation rates. The complete proofs are deferred to~\autoref{Appendix: Auxiliary Lemma and Proofs} for clarity and to maintain the flow of the main text.

\textbf{The Proof Sketch of~\autoref{theorem:main result sobolev} and~\autoref{corollary: main result Holder}.} The key idea of the proof lies in leveraging the composite structure of functions defined on the manifold, specifically the decomposition $f(x) = \sum_{i=1}^K(\rho_i f)\circ \psi_i^{-1} \circ \psi_i(x)$ over a finite atlas of $L$ charts. The approximation strategy is designed to achieve simultaneous approximation of the function and its derivatives. This proceeds in two main stages. First, we separately construct neural networks that simultaneously approximate the coordinate mappings $\{\psi_i(x)\}_{i=1}^K$ and the dimension-reduced functions $\{(\rho_i f)\circ \psi_i^{-1}\}_{i=1}^K$, along with their respective derivatives. For this local approximation, we employ a high-order spline quasi-interpolation scheme, which is then replicated by neural networks and stitched together using a network-based partition of unity. Second, by combining these networks through concatenation and parallel summation, we construct the final global approximation for $f$.

The principal challenge in this argument is that controlling the error of the final network's $s$-th derivative places stricter demands on the approximation of the component functions. Specifically, the Faà di Bruno's formula rule (see~\autoref{Subsection: Sobolev Composite Differences}) implies that to bound the error of the composite function's $s$-th derivative, we must impose accuracy requirements on the inner functions' approximations that are stronger than just controlling their $s$-th derivatives. This is where our use of spline quasi-interpolation becomes crucial. This method allows us to clearly characterize the scale and behavior of the high-order derivatives of the local approximating functions. By constructing neural networks that emulate this quasi-interpolation operator, we can explicitly manage the norms of the network's derivatives. This ensures that the errors accumulated through composition and summation remain tightly bounded, ultimately guaranteeing that the approximation error for each $s$-th order derivative is controlled as required. 

The proof is structured in three parts. First, we establish the core approximation result on the canonical domain of the unit cube $[0,1]^d$. Here, we demonstrate that a specific spline quasi-interpolation operator provides the desired approximation rates and, crucially, can be exactly represented by a $\mathrm{ReLU}^{k-1}$ network with well-defined size and parameter bounds. Second, we generalize this result from the unit cube to arbitrary bounded open domains using a Sobolev extension theorem, which allows us to apply our cube-based construction in a more general setting. Finally, we tackle the case of functions on a smooth manifold by employing a partition of unity to decompose the function into localized components, each of which can be approximated using the results for open domains. The final network is constructed by composing and summing the networks that approximate the localized functions and their corresponding chart maps.

\subsubsection{Approximation on the Unit Cube}\label{Subsubsection: Approximation on the Unit Cube}

In this subsection, we analyze the approximation rates of neural networks for the function class $W_p^k([0,1]^d)$, where $k \in \mathbb{N}_+$ and $1\leq p\leq \infty$. Our strategy is to first establish an approximation theorem for $k$-th order tensor product splines in the Sobolev norm. Subsequently, we leverage the exact representation of these splines by $\mathrm{ReLU}^{k-1}$ networks to derive the simultaneous approximation rates for neural networks on the unit cube $[0,1]^d$.

For a function $f \in W_p^k([0,1]^d)$, we investigate the approximation capability of the tensor product spline quasi-interpolation operator $J$ (defined in \eqref{equation: quasi interpolation operator}) within each cell $T_{\,j_1,\dots,j_d}$ (see~\autoref{lemma: local approximation}).

In~\autoref{Subsection: Averaged Taylor Polynomials}, we established the existence of a polynomial $P$ such that the following inequality holds:
\begin{equation*}
    |f-P|_{W_p^s(T_{\,j_1,\dots,j_d})} \leq C(k,d,\zeta) (\sqrt{d}k)^{(k-s)}N^{-(k-s)} |f|_{W_p^k(T_{\,j_1,\dots,j_d})} \quad s=0,1,\ldots,k.
\end{equation*}
From~\autoref{lemma: local approximation}(ii), we know that the operator $J$ reproduces polynomials, i.e., $JP=P$. Therefore, for any multi-index $\alpha$ with $0\leq|\alpha|\leq s$, we can use the triangle inequality:
\begin{align*}
    & \left\|D^\alpha f-D^\alpha Jf\right\|_{L_p(T_{\,j_1,\dots,j_d})} \\ 
    \leq &\left\|D^\alpha f-D^\alpha P\right\|_{L_p(T_{\,j_1,\dots,j_d})}+\left\|D^\alpha JP-D^\alpha Jf\right\|_{L_p(T_{\,j_1,\dots,j_d})}.
\end{align*}
The first term can be directly estimated as:
\begin{equation*}
    \left\|D^\alpha f-D^\alpha P\right\|_{L_p(T_{\,j_1,\dots,j_d})} \leq C(k,d,\zeta) (\sqrt{d}k)^{(k-|\alpha|)}N^{-(k-|\alpha|)} |f|_{W_p^k(T_{\,j_1,\dots,j_d})}.
\end{equation*}
For the second term, by applying ~\autoref{lemma: local approximation} and the properties of B-spline derivatives, we have:
\begin{align*}
    &\quad\left\|D^\alpha(JP-Jf)\right\|_{L_p(T_{j_1,\dots,j_d})}=  \left\|\sum_{i_1,\ldots,i_d=1}^{N+k-1} J_{i_1,\ldots,i_d}(f-P)D^\alpha B^k_{i_1,\ldots,i_d}(x)\right\|_{L_p(T_{j_1,\dots,j_d})}\\
    &\leq \sum_{i_1,\ldots,i_d=1}^{N+k-1}C(k,d)h_{i_1,\ldots,i_d}^{-\frac{1}{p}}\;\bigl\| f-P \bigr\|_{L_p(T_{j_1,\dots,j_d})} h_{j_1,\dots,j_d}^{\frac{1}{p}}\left\|D^\alpha B^k_{i_1,\ldots,i_d}(x)\right\|_{L_\infty(T_{j_1,\dots,j_d})} \\
    &\leq\sum_{i_1,\ldots,i_d=1}^{N+k-1}C(k,d)N^{-k} |f|_{W_p^k(T_{j_1,\dots,j_d})} \prod_{m=1}^{d}(2N)^{\alpha_m} \frac{(k-1)!}{(k-\alpha_m-1)!}\mathbb{I} \left(\left|i_{m}-j_{m}\right| \leq k\right).\\
    &\leq C(k,d)N^{-(k-|\alpha|)}|f|_{W_p^k(T_{j_1,\dots,j_d})}.
\end{align*} 
Combining the estimates for both terms yields:
\begin{equation*}
    \left\|D^\alpha f-D^\alpha Jf\right\|_{L_p(T_{\,j_1,\dots,j_d})} \leq \frac{C_1|f|_{W_p^k(T_{j_1,\dots,j_d})}}{N^{k-|\alpha|}},
\end{equation*}
where $C_1 = C(k,d,\zeta)\left((\sqrt{d}k)^{(k-|\alpha|)}+(2k + 1)^{2d+|\alpha|-k}9^{d(k-1)}d^{\frac{k}{2}}k^{k+\frac{d}{p}}\right)$.
By taking the $p$-th power of the above inequality and summing over all indices $j_1,\dots,j_d$ and all multi-indices $\alpha$ with $|\alpha|\leq s$, we arrive at the global error bound:
\begin{equation*}
    \|f-Jf\|_{W_p^s([0,1]^d)} \leq \frac{C_1k^{\frac{d}{p}}\binom{s+d}{s}^{\frac{1}{p}}\|f\|_{W_p^k([0,1]^d)}}{N^{k-s}},\quad k\geq s.
\end{equation*}
Furthermore, we need to establish a bound on the higher-order Sobolev norms of the approximant $Jf$, specifically its $W_\infty^m$ norm for $m\in \mathbb{N}_+$. The derivation is analogous. For a multi-index $\beta$ with $|\beta| = m$:
\begin{align*}
    &\quad\left\|D^\beta Jf\right\|_{L_\infty(T_{\,j_1,\dots,j_d})}=  \left\|\sum_{i_1,\ldots,i_d=1}^{N+k-1} J_{i_1,\ldots,i_d}(f)D^\beta B^k_{i_1,\ldots,i_d}(x)\right\|_{L_\infty(T_{\,i_1,\dots,i_d})}\\
    &\leq  \sum_{i_1,\ldots,i_d=1}^{N+k-1}C(k,d)h_{i_1,\ldots,i_d}^{-\frac{1}{p}}\bigl\| f \bigr\|_{L_p\bigl(T_{j_1,\dots,j_d}\bigr)}\left\|D^\beta B^k_{i_1,\ldots,i_d}(x)\right\|_{L_\infty(T_{j_1,\dots,j_d})} \\
    &\leq\sum_{i_1,\ldots,i_d=1}^{N+k-1}C(k,d)N^{\frac{d}{p}}\bigl\| f \bigr\|_{L_p\bigl(T_{j_1,\dots,j_d}\bigr)} \cdot \prod_{m=1}^{d}(2N)^{\beta_m} \frac{(k-1)!}{(k-\beta_m-1)!}\mathbb{I} \left(\left|i_{m}-j_{m}\right| \leq k\right).
\end{align*}
This implies the upper bound:
\begin{equation}\label{equation: derivative bound}
    \|Jf\|_{W^m_\infty([0,1]^d)} \leq C N^{m+\frac{d}{p}}\|f\|_{W_p^k([0,1]^d)}.
\end{equation}
We now proceed to the network construction. We will show that a $\mathrm{ReLU}^{k-1}$ network can exactly represent the spline approximant $Jf$ and, in doing so, precisely characterize how the network size scales with the approximation error. This ultimately yields the simultaneous approximation rates for Sobolev functions on the unit cube. Our construction begins with the exact representation of univariate splines and then generalizes to $d$-dimensional tensor product splines via a product network architecture. Throughout this process, we pay special attention to maintaining control over the infinity norm of the network parameters, i.e., ensuring
\begin{equation*}
\sup_{l=0,1,\ldots,L} \|W_l\|_\infty \lor \sup_{l=1,\ldots,L} \|\mathbf{v}_l\|_\infty \leq 1.
\end{equation*}
The starting point of our construction is the explicit representation of B-spline functions. According to~\cite[Theorem 4.14 and (4.49)]{schumaker2007spline}, the B-splines $B_i^k(x)$ can be expressed as:
\begin{equation}\label{equation: interpolation spline N}
          B^k_i(x) = 
          \left\{
          \begin{aligned}
              &\sum_{j=1}^{k-i+1} \alpha_{ij}(x)^{k-j}_+ + \sum_{j=1}^{i}\beta_{ij}(x-t_{k+j})^{k-1}_+,
              \quad  1 \leq i \leq k-1, \\
              &\frac{N^{k-1}}{(k-1)!}\sum_{j=0}^k(-1)^j\binom{k}{j}\bigg((x-t_i)-\frac{j}{N}\bigg)^{k-1}_+,  
              \quad  k \leq i \leq N, \\
              &\sum_{j=1}^{2N+k-i} \gamma_{ij}(x-t_{i + j -1})^{k-1}_+,
              \quad  N+1 \leq i \leq N+k-1,
          \end{aligned}
          \right.
\end{equation}
where the coefficients satisfy $\max\{|\alpha_{ij}|,|\beta_{ij}|,|\gamma_{ij}|\} \leq N^{k-1}$. Crucially, the truncated power function $(x)_+^{k-1}$ is equivalent to the $\mathrm{ReLU}^{k-1}$ function, which forms the basis of our construction. The following lemma establishes the exact representation of low-degree polynomials using $\mathrm{ReLU}^{k-1}$ functions, with special attention paid to controlling the magnitude of the coefficients.

\begin{lemma}\label{lemma: Decomposition of x^s to ReLU^k}
Let $x\ge0$ and $l,k\in\mathbb{N}^+$ with $l<k$. Let $\{b_{l,i}\}_{i=0}^k$ be an equispaced partition of an interval of length~2, satisfying $\max_{0\le i\le k}|b_{l,i}|\le M$. Then there exist coefficients $a_{l,0},\dots,a_{l,k}$ such that
\begin{equation*}
  x^l = \sum_{i=0}^k a_{l,i}(x + b_{l,i})^k,\quad \forall\,x\ge0,
\end{equation*}
and moreover $  |a_{l,i}| \le M(l,k),i=0,1,\dots,k,$ where $M(l,k)
  = \frac{(k+1)k^k\binom{k}{\lceil k/2\rceil}(M^{k+1}\vee1)}{2^k\lceil k/2\rceil!^{2}\binom{k}{l}}.$
\end{lemma}
\begin{proof}
Expanding both sides in powers of $x$ and matching coefficients yields the linear system
\begin{equation*}
  A\,\mathbf a = \mathbf y,
\end{equation*}
with $A_{m,i}=\binom{k}{m}b_{l,i}^m$, $\mathbf a=(a_{l,0},\dots,a_{l,k})^T$, and
$y_m=\begin{cases}1,&m=k-l,\\0,&m\neq k-l.\end{cases}$
Writing $A=DB$ where
$D=\mathrm{diag}(\binom{k}{0},\dots,\binom{k}{k}),\quad B_{m,i}=b_{l,i}^m,$
and hence
\begin{equation*}
  \mathbf a = A^{-1}\mathbf y = B^{-1}D^{-1}\mathbf y.
\end{equation*}
Since
$(D^{-1}\mathbf y)_m=\begin{cases}1/\binom{k}{l},&m=k-l,\\0,&m\neq k-l.\end{cases}$
We obtain
\begin{equation*}
  \|\mathbf a\|_\infty
  \le \|B^{-1}\|_\infty\;\|D^{-1}\mathbf y\|_\infty
  = \frac{\|B^{-1}\|_\infty}{\binom{k}{l}}.
\end{equation*}
It remains to bound the infinity norm of $B^{-1}$. A classical formula for a Vandermonde inverse gives
$(B^{-1})_{i,m}=\displaystyle\frac{(-1)^{i+m}}{\prod_{p\neq i}(b_{l,i}-b_{l,p})}\,\sigma_m^{(i)},$
where $\sigma_m^{(i)}$ is the elementary symmetric sum of degree $m$ over all nodes except $b_{l,i}$. Since $|b_{l,i}|\le M$, one has
$|\sigma_m^{(i)}| \le \binom{k}{m}\,M^m.$
Because the nodes are equispaced over an interval of length~2,
\begin{equation*}
  \prod_{p\neq i}|b_{l,i}-b_{l,p}|
  \ge \Bigl(\frac{2}{k}\Bigr)^k i!(k-i)!
  \ge \Bigl(\frac{2}{k}\Bigr)^k\lceil\tfrac{k}{2}\rceil!^{2}.
\end{equation*}
Hence,
\begin{equation*}
  \|B^{-1}\|_\infty
  \le \max_i\sum_{m=0}^k
    \frac{\binom{k}{m}M^m}
         {\bigl(\tfrac{2}{k}\bigr)^k\lceil k/2\rceil!^{2}}
  = \frac{(k+1)k^k\binom{k}{\lceil k/2\rceil}(M^{k+1}\vee1)}
         {2^k\lceil k/2\rceil!^{2}}.
\end{equation*}
Combining these estimates yields the claimed bound on $|a_{l,i}|$. This completes the proof.
\end{proof}

We now leverage the preceding lemma to show that $\mathrm{ReLU}^{k-1}$ functions can exactly represent the spline functions $B_i^s$. We demonstrate the case for $s=k$; the proof for $s \neq k$ is analogous. The key idea is that $\mathrm{ReLU}^{k-1}$ can efficiently generate higher-order polynomials, which in turn can represent any lower-order polynomial. By composing the network for $l = \lceil\log s / \log (k-1)\rceil$ layers, we can realize the function $(x)_+^{(k-1)^l}$, where $(k-1)^l \geq s$. We now provide the explicit construction based on the form in \eqref{equation: interpolation spline N}.

First, by applying~\autoref{lemma: Decomposition of x^s to ReLU^k} with $l=2$, $k-1$, and $M=1$, we obtain
\begin{equation*}
  x^2 = \sum_{i=0}^{k-1} a_{2,i}\,(x + b_{2,i})^{k-1},
  \quad
  |a_{2,i}|\le M(2,k-1).
\end{equation*}
For each $i=0,\dots,k-1$, we scale and split the monomial by observing (for odd $k$, the even case is analogous)
\begin{equation*}
  (x + b_{2,i})^{k-1}
  = (x + b_{2,i})_+^{\,k-1} + (-x - b_{2,i})_+^{\,k-1}.
\end{equation*}
Thus, one constructs a small subnetwork computing
\begin{equation*}
  \frac{a_{2,i}}{\lceil M(2,k-1)\rceil}(x + b_{2,i})_+^{\,k-1}
  \quad\text{and}\quad
  \frac{a_{2,i}}{\lceil M(2,k-1)\rceil}(-x - b_{2,i})_+^{\,k-1},
\end{equation*}
and summing over $i$ yields a network
\begin{equation*}
  \mathcal F_{\mathrm{sq}}
  \;\in\;
  \mathcal F\bigl(1,\;2k\,\lceil M(2,k-1)\rceil,\;6k\,\lceil M(2,k-1)\rceil,\;1,\;\infty\bigr)
\end{equation*}
such that $\mathcal F_{\mathrm{sq}}(x)=x^2$. Similarly, setting $l=1$, $k-1$, and $M=1$ in the~\autoref{lemma: Decomposition of x^s to ReLU^k} produces an identity network
\begin{equation*}
  \mathcal F_{\mathrm{id}}
  \;\in\;
  \mathcal F\bigl(1,\;2k\,\lceil M(1,k-1)\rceil,\;6k\,\lceil M(1,k-1)\rceil,\;1,\;\infty\bigr)
\end{equation*}
with $\mathcal F_{\mathrm{id}}(x)=x$.

Finally, to multiply two inputs $x,y$, we use the polarization identity
\begin{equation*}
  xy = \frac14\bigl((x+y)^2 - (x-y)^2\bigr),
\end{equation*}
and combine the square and identity subnetworks accordingly to obtain a network computing the product $xy$. Generalizing this procedure to the product of $n$ terms, we obtain the following lemma.

\begin{lemma}\label{lemma: multiply}
    Let $n \in \mathbb{N}_+, n\geq 2$. For any $x =  (x_1,\ldots,x_n) \in \mathbb{R}^n$, there exists a network
    $$\mathcal{F}_{mult}\in \mathcal{F}\left(\lceil \log n \rceil,k2^{\lceil \log n \rceil}\lceil M(2,k-1)\rceil,k2^{2\lceil \log n \rceil+2}\lceil M(2,k-1) \rceil,1,\infty\right)$$
    that exactly computes the product $\mathcal{F}_{mult}(x) = \prod_{i=1}^nx_i$.
\end{lemma}
\begin{proof}
First, for the two-dimensional case, we have the polarization identity:
$x_1x_2 = \frac{1}{4}((x_1+x_2)^2-(x_1-x_2)^2) = \frac{1}{4}\mathcal{F}_{sq}(x_1+x_2)-\frac{1}{4}\mathcal{F}_{sq}(x_1-x_2)$. This can be implemented by a network with one hidden layer. To compute the product $\prod_{i=1}^n x_i$, we pad the sequence of variables with ones to obtain a product of $2^{\lceil \log n\rceil}$ terms: $x_1x_2\ldots x_n \cdot 1 \cdot \ldots \cdot 1$. Let this new sequence be denoted by $v$. We then perform the following pairwise multiplications recursively for $\lceil \log n \rceil$ steps:
\begin{equation*}
    (v_1,v_2,\ldots,v_{2^{\lceil \log n \rceil}}) \mapsto (v_1v_2,\ldots,v_{2^{\lceil \log n \rceil}-1}v_{2^{\lceil \log n \rceil}}) \mapsto \ldots \mapsto\prod_{i=1}^n x_i 
\end{equation*}
In this process, the maximum width of the network is $k2^{\lceil \log n \rceil+1}\lceil M(2,k-1)\rceil$, and the number of parameters is $k2^{2\lceil \log n \rceil+2}\lceil M(2,k-1) \rceil$. The proof is complete.
\end{proof}

Next, we use~\eqref{equation: interpolation spline N} to accurately represent the B-spline basis function $B^k_i$. For $1\leq i \leq k-1$, by applying~\autoref{lemma: Decomposition of x^s to ReLU^k} with $l=k-j$, we get
\begin{equation*}
    \frac{a_{ij}}{N^{k-1}} (x)_+^{k-j} = \frac{a_{ij}}{N^{k-1}}\sum_{i=0}^{k-1} a_{k-1-j,i} (x+b_{l,i})_+^{k-1},
\end{equation*}
Similar to the previous steps, we construct a network 
\begin{equation*}
    \mathcal{F}_{k-j}\in \mathcal{F}(1,k\lceil M(k-j,k-1)\rceil,3k\lceil M(k-j,k-1)\rceil,1,\infty)
\end{equation*}
to compute $\mathcal{F}_{k-j} (x)= \frac{a_{ij}}{N^{k-1}} (x)_+^{k-j}$. 
Similarly, we compute
\begin{equation*}
    \frac{\beta_{ij}}{N^{k-1}} (x-t_{k+j})_+^{k-1},\quad \frac{(-1)^j}{(k-1)!} (x-t_i-\frac{j}{N})_+^{k-1},\quad \frac{\gamma_{ij}}{N^{k-1}}(x-t_{i+j-1})_+^{k-1}. 
\end{equation*}
All the above networks can be realized with a single hidden layer. For the term $\frac{(-1)^j}{(k-1)!} (x-t_i-\frac{j}{N})_+^{k-1}$, we compute it $\binom{k}{j}$ times and then sum the results. By connecting all these networks in parallel, we obtain
\begin{equation*}
    \mathcal{F}^{k,1}  = (\mathcal{F}^k_1,\ldots,\mathcal{F}^k_{N+k-1}),\quad \mathcal{F}^k_i = \frac{B^k_i(x)}{N^{k-1}},\quad 1\leq i \leq N+k-1.
\end{equation*}
This network satisfies $\mathcal{F}^{k,1} \in \mathcal{F}(1,G_1,S_1,1,\infty)$, with
\begin{equation*}
     G_1 = k\sum_{i=1}^{k-1}\sum_{j=1}^{k-i+1}\lceil M(k-j,k-1)\rceil+k(k-1)+(k+1)2^k(N-k+1),
\end{equation*}
\begin{equation*}
    S_1 = 3k\sum_{i=1}^{k-1}\sum_{j=1}^{k-i+1}\lceil M(k-j,k-1)\rceil+k(k-1)+(k+1)2^k(N-k+1).
\end{equation*}
Next, we construct a network that outputs a constant $N$ (e.g., via $x \mapsto (1,1,\ldots,1) \mapsto N$) and connect it in parallel with the previous network. We also denote this new network by $\mathcal{F}^{k,1}$ for simplicity. Its width is $G_1+N$ and its parameter count is $S_1+2N$. This network computes:
\begin{equation*}
    \mathcal{F}^{k,1} = (N,\frac{B^k_1(x)}{N^{k-1}},\ldots,\frac{B^k_{N+k-1}(x)}{N^{k-1}}).
\end{equation*}
Subsequently, we use the identity network $\mathcal{F}_{id}$ to construct the network
\begin{equation*}
    \mathcal{F}^{k,2}(x) =(N^{k-1},\frac{B^k_1(x)}{N^{k-1}},\ldots,\frac{B^k_{N+k-1}(x)}{N^{k-1}}).
\end{equation*} 
Here,
\begin{align*}
    \mathcal{F}^{k,2} \in& \mathcal{F}(2,G_2:=G_1+N \vee 2k\lceil M(1,k-1)\rceil(N-k+1),\\
    & S_2:= S_1+2N+1+6k(N+k-1)\lceil M(1,k-1)\rceil,1,\infty).
\end{align*}
Then, using the multiplication network from~\autoref{lemma: multiply}, we perform element-wise multiplication to obtain $\mathcal{F}^{k,3} = (B^k_1(x),\ldots,B^k_{N+k-1}(x))$. Here,
\begin{align*}
    \mathcal{F}^{k,3} \in &\mathcal{F}(3,G_3:=G_1\vee2k\lceil M(2,k-1)\rceil(N-k+1),\\
    & S_3:=S_2+\left(16k\lceil M(2,k-1)\rceil(N-k+1)\right),1,\infty )
\end{align*}
We generalize this process to $d$ dimensions. Specifically, we can construct a network
\begin{equation*}
    \mathcal{F}^{k,3}_{d}(x_1,\ldots,x_d) = (B^k_i(x_j)),\quad i=1,\ldots,N+k-1,\quad j=1,\ldots,d.
\end{equation*}
Next, we again apply~\autoref{lemma: multiply} to compute the tensor product splines
\begin{equation*}
 B^k_{i_1,\ldots,i_d}(x) = B^k_{i_1}(x_1)\cdots B^k_{i_d}(x_d), \quad  1\leq i_j \leq N+k-1, j=1,\ldots,d.   
\end{equation*} 
This results in:
\begin{equation*}
    \mathcal{F}^{k,3+\lceil \log d \rceil}_{d}(x)=((N+k-1)^{d},B^k_{i_1,\ldots,i_d}(x)),\quad  1\leq i_j \leq N+k-1,j=1,\ldots,d.
\end{equation*}
Here,
\begin{align*}
    \mathcal{F}^{k,3+\lceil \log d \rceil} \in &\mathcal{F}(3+\lceil \log d \rceil,G_4:=k2^{\lceil \log d \rceil }\lceil M(2,k-1) \rceil(N+k-1)^d,\\
    & S_4:= S_3+(k2^{2\lceil \log d \rceil+2}\lceil M(2,k-1) \rceil+2)(N+k-1)^d).
\end{align*}
Subsequently, we apply a network composed of $q =\lceil \frac{\log C(d)\|f\|_{L_p}/p}{\log (k-1)}\rceil$ layers to compute the required powers, resulting in:
\begin{equation*}
    \mathcal{F}^{k,3+\lceil \log d \rceil+q}_d(x) = ((N+k-1)^{d(k-1)^q},J_{i_1,\ldots,i_d}(f)/(N+k-1)^{d(k-1)^q} B^k_{i_1,\ldots,i_d}(x)).
\end{equation*}
Here,
\begin{equation*}
    \mathcal{F}^{k,3+\lceil \log d \rceil+q}_d \in \mathcal{F}(3+\lceil \log d \rceil+q,G_4,S_5:=S_4+6k\lceil M(1,k-1)\rceil(N+k-1)^d).
\end{equation*}
From~\autoref{lemma: local approximation}, we know that $|J_{i_1,\ldots,i_d}(f)|/(N+k-1)^{d(k-1)^q} < 1$.
Finally, using the multiplication network from~\autoref{lemma: multiply}, we obtain the quasi-interpolant spline:
\begin{equation*}
    \mathcal{F}^{k,4+q}_d(x)=\sum_{i_1,\ldots,i_d=1}^{N+k-1} J_{i_1,\ldots,i_d}(f)B^k_{i_1,\ldots,i_d}(x) = Jf(x).
\end{equation*}
In summary, we have obtained the following lemma, which establishes the simultaneous approximation rates for neural networks on the unit cube $[0,1]^d$.
\begin{lemma}\label{lemma: Sobolev on unit cube}
    Let $d, k \in \mathbb{N}_+$, and $1 \leq p \leq \infty$. For any function $f \in W_p^k([0,1]^d)$, any integer $s$ with $1\leq s < k$, and any $\varepsilon >0$, there exists a $\mathrm{ReLU}^{k-1}$ neural network $g \in \mathcal{F}(L,W,S,B,\infty)$, where $B \in \{1, \infty\}$, such that
    \begin{equation*}
        \|f-g\|_{W^s_p([0,1]^d)} \leq \frac{C}{N^{k-s}}.
    \end{equation*}
    Simultaneously, the network's derivatives are bounded by
    \begin{equation*}
        \|D^\alpha g\|_{L_\infty([0,1]^d)} \leq C N^{|\alpha|+\frac{d}{p}}.
    \end{equation*}
    The network size is bounded by
    \begin{equation*}
       L \leq C,\quad W \leq C\cdot N^{d},  \quad S \leq C\cdot N^{d},
    \end{equation*}
    where the constant $C$ depends on $s, d, f$ and $k$.
\end{lemma}
\begin{remark}
    In summary, the constructed network has a depth of $\mathcal{O}(\log d)$, a width of $\mathcal{O}(N^d)$, and a size of $\mathcal{O}(N^d)$. Regarding the coefficients, the constants in our bounds, such as $M(l,k)$, may appear large. This is primarily an artifact of the strict requirement that all network parameters are bounded by 1. In practice, for analyzing generalization error, one typically bounds network complexity (e.g., via uniform covering numbers) by requiring parameters to lie within a compact set, a less stringent condition. Furthermore, our strategy of using an equidistant partition to represent lower-degree polynomials is a technical choice made for a precise analysis. It is plausible that alternative constructions, for instance using activations like $\mathrm{ReQU} = \mathrm{ReLU}^2$ or $\mathrm{ReLU}^3$ which lead to much simpler multiplication networks, could relax the bounds on the coefficients in our estimates.
\end{remark}

\subsubsection{Extension to Open Domains}\label{Subsubsection: Extension to Open Domains}
In this subsection, we extend our approximation results from the unit cube to general bounded open domains.
\begin{lemma}\label{lemma: Sobolev on open set}
    Let $d, k \in \mathbb{N}_+$, and $1 \leq p \leq \infty$. Let $U$ be a bounded domain in $\mathbb{R}^d$. For any function $f \in W_p^k(U)$, any integer $s$ with $1\leq s < k$, and any $\varepsilon >0$, there exists a neural network $g \in \mathcal{F}(L,W,S,B,\infty)$, where $B \in \{1, \infty\}$, such that
    \begin{equation*}
        \|f-g\|_{W^s_p(U)} \leq \varepsilon.
    \end{equation*}
    Simultaneously, the network's derivatives are bounded by
    \begin{equation*}
        \|D^\alpha g\|_{L_\infty(U)} \leq C \varepsilon^{-\frac{|\alpha|p+d}{(k-s)p}},
    \end{equation*}
    The network size is bounded by
    \begin{equation*}
       L \leq C,\quad W \leq C\cdot\varepsilon^{-d/(k-s)},  \quad S \leq C\cdot\varepsilon^{-d/(k-s)},
    \end{equation*}
    where the constant $C$ depends on $s, d, f$, and $U$.
\end{lemma}

\begin{proof}
    Without loss of generality, we can assume that $U \subset [0,1]^d$. By the Sobolev extension theorem, there exists an extension $\widetilde{f} \in W_p^k(\mathbb{R}^d)$ of $f$ such that $\widetilde{f} = f$ almost everywhere in $U$, and its norm is controlled:
    \begin{equation*}
        \left\|\widetilde{f}\right\|_{W_p^k([-1,1]^d)} \leq \left\|\widetilde{f}\right\|_{W_p^k(\mathbb{R}^d)} \leq C_E \|f\|_{W_p^k(U)},
    \end{equation*}
    where $C_E$ is the norm of the extension operator. We now apply the approximation results for the unit cube to this extension $\widetilde{f}$. From our previous sections, we know there exists a spline quasi-interpolant $J\widetilde{f}$ such that
    \begin{equation*}
        \|\widetilde{f}-J\widetilde{f}\|_{W^s_p([-1,1]^d)} \leq  \frac{C' \|\widetilde{f}\|_{W_p^k([-1,1]^d)}}{N^{k-s}} \leq \frac{C_E C' \|f\|_{W_p^k(U)}}{N^{k-s}}.
    \end{equation*}
    Furthermore, we have shown that $J\widetilde{f}$ can be exactly represented by a neural network $g \in \mathcal{F}(L,W,S,1,\infty)$ with size bounds
    \begin{equation*}
    L \leq C, \quad W \leq CN^{d}, \quad \text{and} \quad S \leq CN^d.
    \end{equation*}
    Restricting to the domain $U$, the approximation error is $\|f-g\|_{W^s_p(U)} \leq \|\widetilde{f}-g\|_{W^s_p([-1,1]^d)}$. To achieve an error of $\varepsilon$, we choose $N$ such that
    \begin{equation*}
        \frac{C_E C' \|f\|_{W_p^k(U)}}{N^{k-s}} \leq \varepsilon \quad \implies \quad N \geq \left( \frac{C_E C' \|f\|_{W_p^k(U)}}{\varepsilon} \right)^{\frac{1}{k-s}}.
    \end{equation*}
    By selecting the smallest integer $N$ that satisfies this condition, we obtain $\|f-g\|_{W_p^k(U)} \leq \varepsilon$. Substituting this choice of $N$ into the network size bounds yields the desired result:
    \begin{equation*}
        L \leq C, \quad W \leq C \cdot \varepsilon^{-d/(k-s)}, \quad \text{and} \quad S \leq C \cdot \varepsilon^{-d/(k-s)}.
    \end{equation*}
    The proof is then finished.
\end{proof}

\subsubsection{Approximation on Manifolds}\label{Subsubsection: Approximation on Manifolds}
In this subsection, we prove our main theorem.
\begin{proof}[Proof of~\autoref{theorem:main result sobolev}]
First, by the properties of the partition of unity on the manifold, we can decompose the target function $f$ as:
\begin{equation*}
    \sum_{i=1}^K (\rho_i f) \circ \psi_i^{-1} \circ \psi_i(x) = f(x), \quad \forall x \in \mathcal{M}.
\end{equation*}
Our approximation strategy is to separately approximate the functions $(\rho_i f) \circ \psi_i^{-1}$ and the coordinate maps $\psi_i$. By assumption, $(\rho_i f) \circ \psi_i^{-1} \in W_p^k(U_i)$, where $U_i$ is the domain of the chart. Therefore, applying~\autoref{lemma: Sobolev on open set}, we know that for any integer $s < k$, there exists a neural network $\widetilde{(\rho_i f) \circ \psi_i^{-1}} \in \mathcal{F}(L_{1,i}, W_{1,i}, S_{1,i}, 1, \infty)$ satisfying
\begin{equation}\label{bounds for f o psi}
    \left\|(\rho_i f) \circ \psi_i^{-1} - \widetilde{(\rho_i f) \circ \psi_i^{-1}}\right\|_{W^s_p(\psi_i(V_i))} \leq \frac{\varepsilon}{4K}, \quad \text{for } i = 1, 2, \dots, K.
\end{equation}
This network also satisfies the derivative bound~\eqref{equation: derivative bound}:
\begin{equation*}
    \left\|\widetilde{(\rho_i f) \circ \psi_i^{-1}}\right\|_{W^{s+1}_\infty(\psi_i(V_i))} \leq C\varepsilon^{-\frac{sp+p+d}{p(k-s)}},
\end{equation*}
and its size is characterized by the following estimates, where the constant $C_{1,i} = C(d, s, k,V_i)$ depends on the indicated parameters:
\begin{equation}\label{equation: net for rhoif o psi}
    L_{1,i} \leq C_{1,i}, \quad W_{1,i} \leq C_{1,i} \cdot \varepsilon^{-d/(k-s)}, \quad S_{1,i} \leq C_{1,i} \cdot \varepsilon^{-d/(k-s)}.
\end{equation}

Next, we approximate the coordinate maps $\psi_i$. According to~\cite{de2021reproducing} Proposition 7(d), when the manifold is complete, we can always choose the open sets $\{V_i\}_{i=1}^K$ to have a sufficiently small radius such that each chart map $\psi_i$ is a local diffeomorphism. By the  Tubular Neighborhood Theorem (see~\cite{lee2003introduction}), there exists a tubular neighborhood $W_i$ of $V_i$ and a corresponding smooth retraction map $\pi: W_i \to V_i$ such that $\pi(x) = x$ for all $x \in V_i$. We then define a smooth extension of $\psi_i$ by $\phi_i(x) = \psi_i(\pi(x))$, so that $\phi_i \in C^{\infty}(W_i)$. In particular, the extension belongs to the Sobolev space $\phi_i \in W^{\lceil s + \frac{(d+p+kp)D}{dp} \rceil}_\infty$. By applying~\autoref{lemma: Sobolev on open set} again, this time to $\phi_i$ and then restricting the domain to $V_i$, we can find a network approximation $\widetilde{\psi_i}$ for $\psi_i$ such that:
\begin{equation}\label{equation: bounds for psi}
    \left\|\psi_i - \widetilde{\psi_i}\right\|_{W^{s}_\infty(V_i)} \leq \frac{\varepsilon^{\frac{d+p+kp}{(k-s)p}}}{4K}, \quad \text{for } i = 1, 2, \dots, K.
\end{equation}
The network size is bounded as follows, with a constant $C_{2,i} = C(d, D, s,V_i, k)$:
\begin{equation}\label{equation: net for psi}
    L_{2,i} \leq C_{2,i}, \quad W_{2,i} \leq C_{2,i} \cdot \varepsilon^{-d/(k-s)}, \quad S_{2,i} \leq C_{2,i} \cdot \varepsilon^{-d/(k-s)}.
\end{equation}

We then compose these two networks to form the approximation $\widetilde{(\rho_i f) \circ \psi_i^{-1}} \circ \widetilde{\psi_i}$. By the triangle inequality and the stability estimate~\eqref{equation: composite difference} for Sobolev norms under composition (see~\autoref{Subsection: Sobolev Composite Differences}), we obtain the following bound on each chart: 
\begin{align*}
    &\left\|(\rho_i f) \circ \psi_i^{-1} \circ \psi_i(x) - \widetilde{(\rho_i f) \circ \psi_i^{-1}} \circ \widetilde{\psi_i}(x)\right\|_{W^s_p(V_i)} \\
    \leq & \left\|(\rho_i f) \circ \psi_i^{-1} \circ \psi_i(x) - \widetilde{(\rho_i f) \circ \psi_i^{-1}} \circ \psi_i(x)\right\|_{W^s_p(V_i)} \\
    + & \left\|\widetilde{(\rho_i f) \circ \psi_i^{-1}} \circ \psi_i(x) - \widetilde{(\rho_i f) \circ \psi_i^{-1}} \circ \widetilde{\psi_i}(x)\right\|_{W^s_p(V_i)} \\
    \leq & C_i(p, V_i)\left(\left\|(\rho_i f) \circ \psi_i^{-1} - \widetilde{(\rho_i f) \circ \psi_i^{-1}}\right\|_{L_p(\psi_i(V_i))} + \varepsilon^{-\frac{sp+p+d}{p(k-s)}}\left\|\psi_i - \widetilde{\psi_i}\right\|_{L_\infty(V_i)}\right).
\end{align*}
By appropriately choosing the precision in the previous bounds (i.e., replacing the original $\varepsilon$ with a smaller value dependent on the constants $C_i,k$), we can ensure that:
\begin{equation*}
    \left\|(\rho_i f) \circ \psi_i^{-1} \circ \psi_i(x) - \widetilde{(\rho_i f) \circ \psi_i^{-1}} \circ \widetilde{\psi_i}(x)\right\|_{W^s_p(V_i)} \leq \frac{\varepsilon}{K}.
\end{equation*}
Combining these results for all $i$, we get the total approximation error on the manifold:
\begin{equation*}
    \sum_{i=1}^K \left\|(\rho_i f) \circ \psi_i^{-1} \circ \psi_i(x) - \widetilde{(\rho_i f) \circ \psi_i^{-1}} \circ \widetilde{\psi_i}(x)\right\|_{W^s_p(\mathcal{M}^d)} \leq \varepsilon.
\end{equation*}

Finally, we describe the architecture of our overall network. For each $i$, we construct the corresponding networks $\widetilde{(\rho_i f) \circ \psi_i^{-1}}$ and $\widetilde{\psi_i}$. Then, using network composition, we connect $\widetilde{(\rho_i f) \circ \psi_i^{-1}}$ and $\widetilde{\psi_i}$. Finally, using parallel connections (i.e., summation), we combine these $k$ subnetworks to obtain the final approximant. Based on this construction and the size estimates from \eqref{equation: net for psi} and \eqref{equation: net for rhoif o psi}, we conclude that there exists a constant $C := C(d, D, s, p, \mathcal{M}^d) = \sum_{i=1}^K (C''_{1,i} + C''_{2,i})$ and a network $g = \sum_{i=1}^K \left(\widetilde{(\rho_i f) \circ \psi_i^{-1}} \circ \widetilde{\psi_i}\right)(x) \in \mathcal{F}(L, W, S, 1, \infty)$ such that
\begin{equation*}
    \|f - g\|_{W^s_p(\mathcal{M}^d)} \leq \varepsilon,
\end{equation*}
and the network complexity is bounded by
\begin{equation*}
    L \leq C, \quad W \leq C \cdot \varepsilon^{-d/(k-s)}, \quad S \leq C \cdot \varepsilon^{-d/(k-s)}.
\end{equation*}
 This completes the proof. The proof of the~\autoref{corollary: main result Holder} is essentially the same as that of~\autoref{theorem:main result sobolev}; the only difference is that the averaged Taylor expansion for the Sobolev case is not required. The proof is omitted.
\end{proof}
\subsection{Lower Bounds}\label{Subsection: Lower Bounds}
This section provides the proofs for~\autoref{theorem: VC-dimension} and~\autoref{theorem: lower bounds}. We begin by introducing some auxiliary notation. Let $\mathcal{F}^{s,k}$ denote a class of neural networks that shares the same architecture as $\mathcal{F}$, with the sole difference being that its activation functions are drawn from the set $\{ \mathrm{ReLU}^{t} : t=s, s+1, \ldots, k\}$. We use $\sigma^{s,k}$ to denote a generic activation function from this set, and note that the specific choice of $\sigma^{s,k}$ can vary from layer to layer. Our first step is to show that any function in the derivative class $D^{\alpha}\mathcal{F}$ (for $|\alpha|\leq s$) can be realized by a network from the class $\mathcal{F}^{k-s-1,k-1}$.

\begin{lemma}\label{lemma: reproduce drivatives}
Let $Q,G,S,s \in \mathbb{N}_+$. For any function $u \in \mathcal{F}(Q,G,S,\infty,\infty)$ and any multi-index $\alpha$ with $|\alpha|\leq s$, we have
\begin{equation}\label{equation: reproduce drivatives}
    D^{\alpha}u \in \mathcal{F}^{k-s-1,k-1}(Q+2,(Q+2)^sG,3^sS,\infty,\infty).
\end{equation}
\end{lemma}

\begin{proof}
We proceed by induction on the derivative order $s$ and the number of layers. First, consider the base case $s=1$. Let $D_i f = \frac{\partial f}{\partial x_i}$ for $i=1,\ldots,d$. Since differentiation commutes with linear transformations, we can, without loss of generality, consider a network with a scalar output and an output scaling factor of $1$.

The claim is straightforward for a single-layer $\mathrm{ReLU}^{k-1}$ network. Now, assume the claim holds for networks with up to $L$ layers. For a network of depth $L+1$ and width $W$, its output is $u^{L+1}(x) = \mathrm{ReLU}^{k-1}\left(\sum_{j=1}^{p_{L+1}}w_j u_j^{L}(x)+b_j\right)$, where $p_{L+1} \le W$. Applying the chain rule, we obtain
$$D_i u^{L+1}(x) = (k-1) \cdot \mathrm{ReLU}^{k-2}\left(\sum_{j=1}^{p_{L+1}}w_j u^{L}(x)+b_j\right) \cdot \sum_{j=1}^{p_{L+1}}w_j D_i u_j^{L}(x).$$
The first term, $\mathrm{ReLU}^{k-2}(\cdot)$, belongs to the class of activations for $\mathcal{F}^{k-2, k-1}$. By the induction hypothesis, the second term $\sum_j w_j D_i u_j^L(x)$, can be realized by a network in $\mathcal{F}^{k-2, k-1}$. By~\autoref{lemma: multiply}, the product of these two functions can be constructed by a new $\mathrm{ReLU}^{k-1}$network. Notice that these two factors belong to the $\mathcal{F}^{k-2,k-1}$ network class with respective widths of $W$ and $(L+2)W$, which implies the resulting product network will have a width of $(L+3)W$. Thus, the claim holds for $s=1$.

To prove the claim for $s\geq 2$, we proceed by induction on $s$. The key observation is that the multiplication sub-network from~\autoref{lemma: multiply} can always be constructed using $\mathrm{ReLU}^{k-1}$ activations. Therefore, the inductive step for higher-order derivatives proceeds in the same manner. Each differentiation step simultaneously introduces a lower-degree activation ($\mathrm{ReLU}^{t-1}$) via the chain rule and necessitates new layers to implement the resulting product. Following our definitions, these combined operations construct a network that provably belongs to the target class $\mathcal{F}^{k-s-1,k-1}$. The conclusion thus follows by induction.
\end{proof}

 Leveraging the reproducing property established in~\autoref{lemma: reproduce drivatives}, we now prove~\autoref{theorem: VC-dimension}, which establishes upper bounds on both the VC-dimension and the pseudo-dimension of $D^\alpha\mathcal{F}$. To this end, we first introduce the following crucial auxiliary lemma.

 \begin{lemma}[Theorem 8.3 in~\cite{bartlett2019nearly}]\label{lemma: partition}
    Let $p_1,\dots,p_m$ be polynomials in $t$ variables of degree at most $q$. Define the sign function as $\mathrm{sgn}(x)\;:=\;\mathbf{1}_{x>0},$ and let
\begin{equation}\label{equation: K upper bound}
    K \;:=\;\bigl|\{\,(\,\mathrm{sgn}(p_1(x)),\dots,\mathrm{sgn}(p_m(x))\,)\colon x\in\mathbb{R}^t\}\bigr|
\end{equation}
be the total number of distinct sign vectors generated by $p_1,\dots,p_m$ over $\mathbb{R}^t$. Then,
\begin{equation*}
    K \le2\Bigl(\tfrac{2e\,m\,q}{t}\Bigr)^t.
\end{equation*}
\end{lemma}

\begin{proof}[Proof of~\autoref{theorem: VC-dimension}]
    By virtue of~\autoref{lemma: reproduce drivatives}, the proof reduces to bounding the VC-dimension of the class $\mathcal{F}^{k-s-1,k-1}(Q+2,(Q+2)^sG,3^sS,\infty,\infty)$. This argument follows directly from the proof of~\cite[Theorem 6]{bartlett2019nearly}. For the sake of readability and to make the paper self-contained, we present the key steps of the proof here. The interested reader is referred to~\cite{bartlett2019nearly} for a complete treatment.

    For notational convenience, let $f(x;\theta)$ denote the output of a network in $\mathcal{F}^{k-s-1,k-1}$ for an input $x\in\mathbb{R}^d$ and a parameter vector $\theta\in\mathbb{R}^{3^sS}$. Thus, each $\theta$ uniquely specifies a function $f(\cdot;\theta)\in\mathcal{F}^{k-s-1,k-1}$. Assume the VC-dimension of this class is $m$. Then there exists a set of points $\{x^1,\dots,x^m\}\subset\mathbb{R}^d$ that is shattered by the network class, meaning:
    \begin{equation*}
        K :=\bigl|\bigl\{\,(\,\mathrm{sgn}(f(x^1;\theta)),\dots,\mathrm{sgn}(f(x^m;\theta))\,)\colon \theta\in\mathbb{R}^{3^sS}\bigr\}\bigr|=2^m.
    \end{equation*}
    The proof hinges on constructing a suitable partition $\{P_1,P_2,\ldots,P_N\}$ of the parameter space $\mathbb{R}^{3^sS}$ such that on each cell $P_i$, the function $f(x^j;\theta)$ becomes a polynomial in the parameters $\theta$. This allows us to apply~\autoref{lemma: partition} to bound the number of sign patterns on each cell, which in turn bounds $K$ and therefore the VC-dimension.

    This partition is constructed recursively, layer by layer. The first layer is linear, so we can take the entire parameter space $\mathbb{R}^{3^sS}$ as the initial partition. For subsequent layers, assuming we have a valid partition for layer $n-1$, we refine it for layer $n$. The zero-sets of the pre-activation inputs at layer $n$ define a set of hyperplanes. These hyperplanes are used to further divide the cells of the existing partition. Within each new, smaller cell, the sign of every pre-activation is fixed, causing the activated output to be a polynomial. The number of new cells generated by this refinement at each layer is bounded using ~\autoref{lemma: partition}.

    This recursive construction yields a final partition that satisfies the desired polynomial property, with an exponential bound on the total number of cells. Specifically, we can obtain the following bounds. The number of sign patterns is bounded by summing over the patterns in each cell:
    \begin{equation*}
        K \leq \sum_{i=1}^N \bigl|\bigl\{\,(\,\mathrm{sgn}(f(x^1;\theta)),\dots,\mathrm{sgn}(f(x^m;\theta))\,)\colon \theta\in P_i\bigr\}\bigr|\leq N \cdot 2\left(\frac{2em(1+(Q+1)k^{Q+1})}{3^sS}\right)^{3^sS},
    \end{equation*}
    and the total number of partition cells, $N$, is bounded by:
    \begin{equation*}
        N \leq \prod_{i=1}^{Q+1}2\left(\frac{2emk_i(1+(i-1)k^{i-1})}{S_i}\right)^{S_i},
    \end{equation*}
    where $k_i$ is the number of units in layer $i$, and $S_i$ is the cumulative number of parameters up to layer $i$.

    Combining these two inequalities and substituting $K=2^m$, then taking the logarithm of both sides, yields the desired bound:
    \begin{equation*}
        m \lesssim s3^sSQ\log S.
    \end{equation*}
    The proof for the pseudo-dimension bound is analogous. The only modification required is to adapt the definition of $K$ to count pseudo-shattered patterns instead of shattered patterns. The subsequent calculations follow the same logic and are omitted here. The proof is complete.
    \end{proof}

We now provide the proof for the lower bound stated in~\autoref{theorem: lower bounds}. This establishes the necessity of the parameter count for any $\mathrm{ReLU}^k$ network class to achieve the desired simultaneous approximation accuracy.

\begin{proof}[Proof of~\autoref{theorem: lower bounds}]
The core idea of the proof, inspired by~\cite{yarotsky2017error}, is to construct a large family of functions that are hard to approximate, thereby forcing any successful network to have high complexity.

First, by standard results on manifold geometry, we can select a set of points $\mathcal{G}=\{x_1,x_2,\ldots,x_{M}\}$ that is $\frac{1}{2N}$-separated and $\frac{1}{N}$-dense with respect to the geodesic distance. The cardinality of this set is bounded by $M \gtrsim N^d / C(\mathcal{M}^d)$. Associated with this set, we can construct a smooth partition of unity $\{\rho_i\}_{i=1}^M$. Using standard mollification techniques (e.g.,~\cite{evans2022partial}), each function $\rho_i$ can be designed to satisfy: $\rho_i(x) = 1$ for $x \in B_{\mathcal{M}^d}(x_i,\frac{1}{2N})$, $\rho_i(x) = 0$ for $x \notin B_{\mathcal{M}^d}(x_i,\frac{1}{N})$, and $0 < \rho_i(x) < 1$ otherwise.

Next, we pull back a canonical Euclidean bump function onto the manifold. Let $\phi \in C_{\mathrm{c}}^{\infty}\bigl(B_{\mathbb{R}^d}(0,\tfrac{1}{2})\bigr)$ be a smooth bump function such that for a specific multi-index $\alpha_0$ with $|\alpha_0|=s$, we have $D^{\alpha_0}\phi(0)=1$. We define a localized version of this function around each point $x_i$ as $\phi_i(x) = \phi(N\cdot\exp^{-1}_{x_i}(x))$. For any binary vector $z = (z_1, \dots, z_M) \in \{0,1\}^M$, we construct the function:
\begin{equation*}
    f_z(x) = \sum_{i=1}^M z_iN^{-k} \rho_i(x) \phi_i(x)
\end{equation*}
By construction, these functions satisfy $D_{\text{chart}}^{\alpha_0}f_z(x_i) = z_i N^{-(k-s)}$, where $D_{\text{chart}}^{\alpha_0}$ denotes differentiation in the local chart coordinates around $x_i$. It is straightforward to verify that $f_z \in\mathcal{CH}^k(\mathcal{M}^d)$ with a norm bounded independently of $z$.

By the theorem's premise, for any such $f_z$, there exists a network $g_z$ from our class that achieves an $s$-th order simultaneous approximation with error $\epsilon$. Let us choose $N$ such that $\epsilon = \frac{1}{3}N^{-(k-s)}$. This implies that the derivatives of the approximating network $g_z$ must satisfy:
\begin{equation*}
    D_{\text{chart}}^{\alpha_0}g_z(x_i) 
    \begin{cases}
        \in [N^{-(k-s)} - \epsilon, N^{-(k-s)} + \epsilon] \implies D_{\text{chart}}^{\alpha_0}g_z(x_i) > \frac{1}{2}N^{-(k-s)}, \quad &\text{if}\, z_i=1 \\
        \in [-\epsilon, \epsilon] \implies D_{\text{chart}}^{\alpha_0}g_z(x_i) < \frac{1}{2}N^{-(k-s)}, \quad &\text{if}\, z_i=0
    \end{cases}
\end{equation*}
This shows that we can distinguish the value of $z_i$ by thresholding the network's derivative at $\frac{1}{2}N^{-(k-s)}$. Therefore, the class of functions $\{D_{\text{chart}}^{\alpha_0}g : g \in \mathcal{F}\}$ can pseudo-shatter the set $\{x_1, \dots, x_M\}$. This gives a lower bound on the pseudo-dimension: $\mathrm{PDim}(D_{\text{chart}}^{\alpha_0}\mathcal{F}) \geq M$. Expressing this in terms of $\epsilon$, we get $\mathrm{PDim}(D_{\text{chart}}^{\alpha_0}\mathcal{F}) \gtrsim \epsilon^{-\frac{d}{k-s}}$.

A crucial step is to relate this chart-based derivative to the ambient derivative of the network. From differential geometry, derivatives in local coordinates (e.g., geodesic normal coordinates) can be expressed as a linear combination of derivatives in the ambient Euclidean space:
\begin{equation*}
    D^{\beta}_{\text{chart}}\,g(x_m)=\sum_{|\alpha|=|\beta|}\!  c_{\beta,\alpha}(x_m)\, D^{\alpha}_{\text{ambient}}\,g(x_m).
\end{equation*}
The coefficients $c_{\beta,\alpha}(x_m)$ depend smoothly on the point $x_m$ via the Jacobian of the exponential map. This linear transformation does not significantly increase the complexity, as it only introduces a fixed number of sign changes. Therefore, the pseudo-dimension of the ambient derivative class is of the same order.

Finally, we connect this complexity lower bound to the number of network parameters. From our upper bound in~\autoref{theorem: VC-dimension}, we know that $\mathrm{PDim}(D^{\alpha_0}\mathcal{F}) \lesssim SL\log S$. Combining this with our newly derived lower bound gives:
\begin{equation*}
    SQ\log S \gtrsim \mathrm{PDim}(D^{\alpha_0}\mathcal{F}) \gtrsim \epsilon^{-\frac{d}{k-s}}.
\end{equation*}
This completes the proof.
\end{proof}
\newpage
\section*{Appendix}
\addcontentsline{toc}{section}{Appendix}
\label{allapp}
In this section,~\autoref{Appendix: Notations} provides precise definitions of the notation used throughout this paper, and~\autoref{Appendix: Auxiliary Lemma and Proofs} presents the auxiliary lemmas and their corresponding proofs.
\appendix
\section{Notations}\label{Appendix: Notations}
\subsection{Function Space}
We introduce the function spaces $ L_p(\Omega) $ and the Sobolev spaces $ W_p^k(\Omega) $. For $ 1 \leq p \leq \infty $, the $ L_p $ norm is defined as
\begin{equation*}
    \|f\|_{L_p} := \begin{cases}
    \left( \int_{\Omega} |f(x)|^p \, dx \right)^{1/p}, & 1 \leq p < \infty, \\
    \operatorname*{ess\,sup}_{x \in \Omega} |f(x)|, & p = \infty.
    \end{cases} 
\end{equation*}
The Sobolev norm is defined as
\begin{equation*}
    \|f\|_{W_p^k} := \begin{cases}
    \left( \sum_{|\alpha| \leq k} \|D^\alpha f\|_{L_p}^p \right)^{1/p},  &1 \leq p < \infty, \\
    \max_{|\alpha| \leq k} \|D^\alpha f\|_{L_\infty}, & p = \infty.
    \end{cases}
\end{equation*}
The $ D^\alpha f $ denotes the $ \alpha $-th weak derivative of $ f $. The function spaces $ L_p(\Omega) $ and $ W_p^k(\Omega) $ consist of functions with finite norms. Additionally, we denote
\begin{equation*}
    |f|_{W_p^k} := 
    \begin{cases}
        \left( \sum_{|\alpha| = k} \|D^\alpha f\|_{L_p}^p \right)^{1/p} ,  &1 \leq p < \infty,\\
         \max_{|\alpha| = k} \|D^\alpha f\|_{L_\infty}, & p = \infty.
    \end{cases}
\end{equation*}
\subsection{Riemannian Manifolds}
Our notation and terminology primarily follow~\cite{lee2003introduction}. In this paper, we consider a compact, connected, and complete $ d $-dimensional Riemannian manifold $ (\mathcal{M}^{d}, g) $ with bounded geometry, where the metric tensor $ g $ is a positive definite second-order covariant tensor field. For any point $ x \in \mathcal{M}^{d} $, $ T_x\mathcal{M}^{d} $ denotes its tangent space, and the tangent bundle $ T\mathcal{M}^{d} := \bigcup_{x \in \mathcal{M}^{d}} T_x\mathcal{M}^{d} $ along with the cotangent bundle $ T^*\mathcal{M}^{d} $ constitute the fundamental geometric structures on the manifold. In a local coordinate system $ (x^i) $, the natural basis for the tangent space is $ \partial_j := \frac{\partial}{\partial x^j} $, and the dual basis for the cotangent space is $ dx^i $. Using the Einstein summation convention, the metric tensor is expressed as $ g = g_{ij} \, dx^i \, dx^j $, where $ g_{ij} = g(\partial_i, \partial_j) $. The inner product of any tangent vectors $ X_x, Y_x \in T_x\mathcal{M}^{d} $ is defined as $ \langle X_x, Y_x \rangle_{g_x} := g_x(X_x, Y_x) $, and the inverse matrix of the metric matrix $ (g_{ij}) $ is denoted by $ (g^{ij}) $. 

The volume element on the manifold is denoted by $ dV_g $, and in local coordinates, it is expressed as $ dV_g = \sqrt{|g|} \, dx^1 \wedge \dots \wedge dx^d $, where $ |g| = \det(g_{ij}) > 0 $. For a smooth function $ f: \mathcal{M}^{d} \rightarrow \mathbb{R} $, its differential satisfies $ df_x(X_x) = X_x(f) $, where $ X_x $ is an arbitrary tangent vector. The gradient of a function $ h $, denoted by $ \mathrm{grad}_g h $, is defined via the isomorphism induced by the Riemannian metric $ T\mathcal{M}^{d} \cong T^*\mathcal{M}^{d} $, and its local expression is $ \mathrm{grad}_g h = g^{ij} \frac{\partial h}{\partial x^j} \, \partial_i $. The Levi-Civita connection $ \nabla $, which is compatible with the metric, satisfies $ \nabla g = 0 $. Specifically, for a vector field $ X = X^j \partial_j $, its covariant derivative is given by $ \nabla_j X^i = \partial_j X^i + \Gamma^i_{kj} X^k $, where the Christoffel symbols are $ \Gamma^i_{kj} = \frac{1}{2} g^{il} \left( \partial_k g_{lj} + \partial_j g_{lk} - \partial_l g_{kj} \right) $. The divergence of a vector field $ X $ is defined as $ \mathbf{div}_g X := \nabla_i X^i = \partial_i X^i + \Gamma^i_{ki} X^k $. 

For a smooth function on a Riemannian manifold $\mathcal{M}^d$, the Laplace-Beltrami operator $\Delta_{\mathcal{M}}$ is defined as
\begin{equation*}
    \Delta_{\mathcal{M}} f := -\mathrm{div}_g(\nabla f).
\end{equation*}
In local coordinates, this operator can be expressed as
\begin{equation*}
     \Delta_{\mathcal{M}} f = -\frac{1}{\sqrt{\det(g)}} \partial_i \left(g^{ij} \sqrt{\det(g)} \partial_j f\right).
\end{equation*}
On a complete manifold, it is a classical result that $\Delta_{\mathcal{M}}$, initially defined on the space of smooth, compactly supported functions, is essentially self-adjoint. It therefore admits a unique self-adjoint extension to an (unbounded) operator on $L^2(M)$, which we also denote by $\Delta_{\mathcal{M}}$. By the spectral theorem, its spectrum satisfies $\sigma(\Delta_{\mathcal{M}}) \subset [0, \infty)$, and there exists a projection-valued measure $P$ on $\sigma(\Delta_{\mathcal{M}})$ such that
\begin{equation*}
    \Delta_{\mathcal{M}} = \int_{\sigma(\Delta_{\mathcal{M}})} \lambda \, dP(\lambda).
\end{equation*}
For any $f, g \in L^2(M)$, this induces a bounded complex measure
\begin{equation*}
    P_{f,g}(E) = \langle P(E)f, g \rangle, \quad \text{where } E \subset [0, \infty) \text{ is a Borel set}.
\end{equation*}
Furthermore, for any Borel-measurable function $\Phi: [0, \infty) \to \mathbb{R}$, the operator $\Phi(\Delta_{\mathcal{M}})$ is defined via the functional calculus, satisfying
\begin{equation*}
    \langle \Phi(\Delta_{\mathcal{M}})f, g \rangle = \int_0^{\infty} \Phi(\lambda) \, dP_{f,g}(\lambda)
\end{equation*}
for all $f \in \mathrm{dom}(\Phi(\Delta_{\mathcal{M}}))$ and $g \in L^2(M)$. The domain of this operator is given by
\begin{equation*}
    \mathrm{dom}(\Phi(\Delta_{\mathcal{M}})) = \left\{ f \in L^2(M) \, \Big| \, \int_0^{\infty} \Phi(\lambda)^2 \, dP_{f,f}(\lambda) < \infty \right\}.
\end{equation*}
Finally, we provide the definition of bounded geometry, a key assumption for our analysis.
\begin{definition}[Bounded geometry]\label{def:bounded_geometry}
    A complete Riemannian manifold $\mathcal{M}^d$ is said to have bounded geometry if it satisfies the following two conditions:
    \begin{enumerate}[(i)]
        \item  The injectivity radius is uniformly bounded below by a positive constant. That is, there exists a constant $\delta > 0$ such that
    $$ \mathrm{inj}(\mathcal{M}^d) \ge \delta. $$
    This ensures that for any point $x \in \mathcal{M}^d$, the exponential map $\exp_x$ is a diffeomorphism from the ball $B(0, \delta) \subset T_x\mathcal{M}^d$ to its image in $\mathcal{M}^d$.
    \item The components of the metric tensor and all their derivatives are uniformly bounded in normal coordinates. Specifically, there exist constants $C > 0$ and $\{C_\alpha\}_{\alpha}$ (independent of the choice of coordinates) such that in any normal coordinate chart, the following hold for the metric components $g_{ij}$:
    $$ \det(g_{ij}) \ge C \quad \text{and} \quad |\partial^\alpha g_{ij}| \le C_\alpha $$
    for all multi-indices $\alpha$.
    \end{enumerate}
\end{definition}
\section{Auxiliary Lemmas and Proofs}\label{Appendix: Auxiliary Lemma and Proofs}
\subsection{Tensor Product Splines}\label{Subsection: Tensor Product Splines}

For the sake of self-containment, we provide the definition of tensor product splines~\cite{schumaker2007spline}. The construction begins with the univariate spline basis, starting with the normalized B-spline of order $k$.

Let the functions $u_1, u_2, \ldots, u_k$ be defined on an interval $ I \subset \mathbb{R} $, and let a set of points $t_1 \le t_2 \le \cdots \le t_k$ be given in this interval. Suppose these $k$ points can be written as
\[
t_1, t_2, \ldots, t_k = \underbrace{\tau_1, \ldots, \tau_1}_{k_1 \text{ times}},\; \ldots,\; \underbrace{\tau_l, \ldots, \tau_l}_{k_l \text{ times}},
\]
where each $\tau_i$ appears with multiplicity $k_i$, satisfying $\sum_{i=1}^{l} k_i = k$. Let $D^i u$ denote the $i$-th derivative of a function $u$. We define the following determinant:
\[
\begin{aligned}
D 
\begin{pmatrix}
t_1, \cdots, t_k \\
u_1, \cdots, u_k
\end{pmatrix} = \det
\begin{bmatrix}
u_1(\tau_1) & u_2(\tau_1) & \cdots & u_k(\tau_1) \\
Du_1(\tau_1) & Du_2(\tau_1) & \cdots & Du_k(\tau_1) \\
\vdots & \vdots & \ddots & \vdots \\
D^{k_1-1}u_1(\tau_1) & D^{k_1-1}u_2(\tau_1) & \cdots & D^{k_1-1}u_k(\tau_1) \\
\vdots & \vdots & \ddots & \vdots \\
u_1(\tau_l) & u_2(\tau_l) & \cdots & u_k(\tau_l) \\
Du_1(\tau_l) & Du_2(\tau_l) & \cdots & Du_k(\tau_l) \\
\vdots & \vdots & \ddots & \vdots \\
D^{k_l-1}u_1(\tau_l) & D^{k_l-1}u_2(\tau_l) & \cdots & D^{k_l-1}u_k(\tau_l)
\end{bmatrix}.
\end{aligned}
\]
Given a positive integer $k>0$, the $k$-th order divided difference of a function $f$ with respect to the points $t_1, \ldots, t_{k+1}$ on the interval $I$ is defined as
\[
[t_1,\cdots,t_{k+1}]f = \frac
{D
\begin{pmatrix}
t_1, t_2, \cdots,t_{k}, t_{k+1} \\
1, x, \cdots, x^{k-1}, f
\end{pmatrix}
}
{D
\begin{pmatrix}
t_1, t_2, \cdots,t_{k}, t_{k+1} \\
1, x, \cdots, x^{k-1}, x^k
\end{pmatrix}}.
\]
Consider a non-decreasing sequence of real numbers, called knots:
\[
\cdots \le t_{-1} \le t_{0} \le t_1 \le \cdots.
\]
For any integer $i$ and positive integer $k>0$, the normalized B-spline of order $k$ corresponding to the knots $t_i, t_{i+1}, \ldots, t_{i+k}$ is defined as
\[
B_i^k(x) = (-1)^k\,(t_{i+k}-t_{i})\,[t_i,\cdots,t_{i+k}](x-t)_+^{k-1},
\]
where $(z)_+ = \max(z,0)$ is the truncated power function. For $N>0$, we choose the following augmented uniform partition:
\[t_1= \cdots = t_{k} = 0 < t_{k+1} < \cdots < t_{N+k-1} < 1= t_{N+k} = \cdots =t_{N+2k-1}.\]
The spline space is then defined as $\mathrm{span}\{B_i^k\}_{i=1}^{N+k-1}$. Following~\cite{schumaker2007spline}, we can construct a dual basis of linear functionals $\{J_i\}_{i=1}^{N+k-1}$ satisfying $J_i (B_j^k) = \delta_{ij}$. This allows us to define the spline quasi-interpolation operator:
\begin{equation*}
    Jf(x) = \sum_{i=1}^{N+k-1} J_i(f) B_i^k(x),\quad  x \in [0,1].
\end{equation*}
Building upon this, we introduce tensor product splines and quasi-interpolation operator on $[0,1]^d$. For each dimension $j=1, \ldots, d$, we select an analogous augmented uniform partition $\{t_{j,i}\}_{i=1}^{N+2k-1}$ and obtain $d$ sets of spline bases $\{ B_{i_j}^{k}(x_j)\}_{i_j=1}^{N+k-1}$. The tensor product spline space is 
\begin{equation*}
    \mathcal{S} = \mathrm{span}\{B_{i_1}^{k}(x_1)\cdots B_{i_d}^{k}(x_d)\}_{i_1,\ldots,i_d=1}^{N+k-1}.
\end{equation*}
The tensor product B-splines are defined as:
\begin{equation}\label{equation: definition of tensor product B-splines}
    B^k_{i_1,\ldots,i_d}(x) = B^k_{i_1}(x_1)\cdots B^k_{i_d}(x_d),\quad x= (x_1,\ldots,x_d)\in [0,1]^d,
\end{equation}
with support $\mathrm{supp}(B^k_{i_1,\ldots,i_d}) = \bigotimes_{j=1}^d \bigl[\,t_{j,i_j},\;t_{j,i_j+k}\bigr]$. The corresponding dual basis for the tensor product spline is given by:
\begin{equation*}
    J_{i_1,\ldots,i_d} = J_{i_1} \otimes \cdots \otimes J_{i_d},
\end{equation*}
which satisfies the property
\begin{equation*}
    J_{i_1,\ldots,i_d}(B^k_{j_1,\ldots,j_d}) = J_{i_1}(B^k_{j_1})\cdots J_{i_d}(B^k_{j_d}) = \delta_{i_1,j_1}\cdots\delta_{i_d,j_d}.
\end{equation*}
The tensor product spline quasi-interpolation operator is then defined as:
\begin{equation}\label{equation: quasi interpolation operator}
    Jf(x) = \sum_{i_1,\ldots,i_d=1}^{N+k-1} J_{i_1,\ldots,i_d}(f)B^k_{i_1,\ldots,i_d}(x).
\end{equation}

This quasi-interpolation operator possesses several favorable properties. For instance,~\cite{schumaker2007spline} shows that the operator is bounded and has a reproducing property on the spline space, as summarized below.

\begin{lemma}[Section 12.2 in \cite{schumaker2007spline}] \label{lemma: local approximation}
Let $B^k_{i_1,\ldots,i_d}$ be the normalized tensor product B-splines defined in~\eqref{equation: definition of tensor product B-splines}, let $\mathcal{S}$ be the spline space they span, and let $J$ be the corresponding quasi-interpolation operator. Then the following properties hold:
\begin{enumerate}[(i)]
    \item Let $T_{\,i_1,\dots,i_d} = \bigotimes_{j=1}^d \bigl[\,t_{j,i_j},\;t_{j,i_j+k}\bigr]$ and $h_{\,i_1,\dots,i_d} = \prod_{j=1}^{d} \Bigl(t_{j,i_j+k} - t_{j,i_j}\Bigr)$. For any $f \in L_p\bigl({T}_{\,i_1,\dots,i_d}\bigr)$, the following inequality holds:
    \begin{equation*}
        \bigl|J_{i_1\cdots i_d} (f)\bigr|\leq(2k + 1)^d\,9^{d(k-1)}\,h_{i_1\cdots i_d}^{-\frac{1}{p}}\;\bigl\| f \bigr\|_{L_p\bigl(T_{\,i_1,\dots,i_d}\bigr)}.
    \end{equation*}
    Consequently, $J:L_p([0,1]^d) \mapsto L_p([0,1]^d)$ is a bounded linear operator. That is, there exists a constant $C>0$, independent of $f$, such that
    \begin{equation*}
        \bigl\|Jf\bigr\|_{L_p([0,1]^d)} \leq C\bigl\|f\bigr\|_{L_p([0,1]^d)}.
    \end{equation*}
    \item $J$ has the reproducing property on $\mathcal{S}$, i.e.,
    \begin{equation*}
        Js = s \quad \forall s \in \mathcal{S}.
    \end{equation*}
\end{enumerate}
\end{lemma}
The proof can be found in~\cite[Theorem 12.5 and Theorem 12.6]{schumaker2007spline}.

Finally, since simultaneous approximation involves derivatives, we provide an upper bound for the derivatives of the B-spline basis functions. This bound is highly dependent on the choice of knots.

\begin{lemma}\label{lemma: control derivatives}
Let $B_i^s(x)$ be the normalized B-spline basis of order $s$ as defined above. For any integer $\sigma$ with $1 \leq \sigma < s$, the following inequality holds:
\begin{equation*}
   \max_{1\leq i \leq N+2k-s-1 } \|D^\sigma B_i^s \|_{L_\infty([0,1])} \leq (2N)^\sigma \frac{(s-1)!}{(s-\sigma-1)!}.
\end{equation*}
\end{lemma}
This proof follows directly from~\cite[Theorem 4.22]{schumaker2007spline}.
\subsection{Stability Estimates}\label{Subsection: Sobolev Composite Differences}
The central goal of this subsection is to establish an upper bound for $\|f(g(x))-f(h(x))\|_{W^n_\infty}$. Our strategy is to first derive the result for smooth functions, i.e., $f,g \in C^{\infty}$, and then extend it to general Sobolev functions via a density argument. To do this, we require the Faà di Bruno formula for higher-order derivatives of composite functions.

We begin by defining the incomplete Bell polynomials $ B_{n,k}(x_1, x_2, \ldots, x_{n-k+1}) $:
\begin{equation}\label{equation: one Bell}
    B_{n,k} = \sum \frac{n!}{m_1!\,m_2!\,\cdots\,m_{n-k+1}!} \left(\frac{x_1}{1!}\right)^{m_1} \left(\frac{x_2}{2!}\right)^{m_2} \cdots \left(\frac{x_{n-k+1}}{(n-k+1)!}\right)^{m_{n-k+1}},
\end{equation}
where the sum is taken over all sequences of non-negative integers $ m_1, m_2, \ldots, m_{n-k+1} $ that satisfy:
\[
\sum_{i=1}^{n-k+1} m_i = k \quad \text{and} \quad \sum_{i=1}^{n-k+1} i \cdot m_i = n.
\]

\begin{definition}[Univariate Faà di Bruno's Formula]
    For univariate functions $x \in \mathbb{R}$, $f:\mathbb{R}\mapsto \mathbb{R}$, and $g:\mathbb{R}\mapsto \mathbb{R}$, assuming both $f$ and $g$ are $n$-times differentiable, the $n$-th derivative of their composition is given by
    \begin{equation*}
        D^n f(g(x)) = \sum_{k=1}^n f^{(k)}(g(x)) \cdot B_{n,k}\left(g'(x), g''(x), \dots, g^{(n-k+1)}(x)\right),
    \end{equation*}
    where $ B_{n,k} $ are the Bell polynomials as defined in~\eqref{equation: one Bell}.
\end{definition}

To state the multivariate Faà di Bruno formula, we first define the corresponding tensor Bell polynomials. Let $x=\{x_\beta\}_{1\le|\beta|\le|\alpha|-k+1}$ be a set of formal variables. The tensor Bell polynomial is defined as:
\begin{equation*}
    B_{\alpha,k}(x)=\sum\frac{\alpha!}{\displaystyle\prod_{1\le|\beta|\le|\alpha|-k+1} m_\beta!}\,\prod_{1\le|\beta|\le|\alpha|-k+1}\left(\frac{x_\beta}{\beta!}\right)^{\,m_\beta}.
\end{equation*}
The sum is taken over all maps $m:\{\beta\in\mathbb{N}_0^D: 1\le|\beta|\le|\alpha|-k+1\}\to\mathbb{N}_0$ satisfying the conditions:
\[
    \sum_{1\le|\beta|\le|\alpha|-k+1} m_\beta = k,\quad \text{and} \quad \sum_{1\le|\beta|\le|\alpha|-k+1} m_\beta\, \beta = \alpha.
\]
Here, the condition $\sum m_\beta\, \beta = \alpha$ means that for each component $ i=1,\dots,D $, we have
\[
\sum_{|\beta|\ge1} m_\beta\, \beta_i = \alpha_i.
\]
Now we can present the multivariate Faà di Bruno formula.
\begin{definition}[Multivariate Faà di Bruno's Formula]
    Let $x \in [0,1]^D$, $g: \mathbb{R}^D \to \mathbb{R}^d$, and $f: \mathbb{R}^d \to \mathbb{R}$. Assuming $f$ and $g$ are smooth ($C^{\infty}$), then for any non-zero multi-index $\alpha$, we have:
\end{definition}
\begin{equation*}
    D^\alpha \Bigl( f\circ g \Bigr)(x)=\sum_{k=1}^{|\alpha|} D^k f\bigl(g(x)\bigr)\Biggl(B_{\alpha,k}\Bigl( \Bigl\{\frac{D^\beta g(x)}{\beta!}\Bigr\}_{1\le|\beta|\le |\alpha|-k+1} \Bigr)
\Biggr).
\end{equation*}
Here, the Bell polynomial term is explicitly
\begin{equation*}
    B_{\alpha,k}\Biggl(\left\{\frac{D^\beta g}{\beta!}\right\}_{1\le|\beta|\le|\alpha|-k+1}\Biggr)=\sum \frac{\alpha!}{\displaystyle\prod_{1\le|\beta|\le|\alpha|-k+1} m_\beta!}\prod_{1\le|\beta|\le|\alpha|-k+1} \left(\frac{D^\beta g}{\beta!}\right)^{m_\beta},
\end{equation*}
with the summation defined as before. For $|\alpha| \leq n$, we apply the multivariate Faà di Bruno formula. Let $E(x) = f(g(x))-f(h(x))$. Then we have
\begin{equation*}
    D^\alpha E =  \sum_{k=1}^{|\alpha|}\Biggl\{ \left[D^k f\circ g-D^k f\circ h\right]\, B_{\alpha,k}\Bigl(\Bigl\{ \frac{D^\beta g}{\beta!}\Bigr\}\Bigr) 
+ D^k f\circ h \,\Delta B_{\alpha,k} \Biggr\},
\end{equation*}
where
\[
\Delta B_{\alpha,k}(x)=B_{\alpha,k}\Bigl(\Bigl\{ \frac{D^\beta g(x)}{\beta!}\Bigr\}\Bigr)-B_{\alpha,k}\Bigl(\Bigl\{ \frac{D^\beta h(x)}{\beta!}\Bigr\}\Bigr).
\]
By the Mean Value Theorem, we can bound the first term in the sum:
\begin{equation*}
\begin{aligned}
        \left\|E_1(x)\right\|_{L_\infty}:&= \left\|\sum_{k=1}^{|\alpha|}\left[D^k f\bigl(g(x)\bigr)-D^k f\bigl(h(x)\bigr)\right]\, B_{\alpha,k}\Bigl(\Bigl\{ \frac{D^\beta g(x)}{\beta!}\Bigr\}\Bigr)\right\|_{L_\infty}\\
        &\leq \sum_{k=1}^{|\alpha|}C(\alpha,k)\left\|f\right\|_{W_\infty^{n+1}}\left\|g-h\right\|_{L_\infty}\|g\|_{W^n_\infty}^k.
\end{aligned}
\end{equation*}
Here, $C(\alpha,k)= \sum_{m_\beta}\frac{\alpha!}{\displaystyle\prod m_\beta!}\prod_{\beta}\left(\frac{1}{\beta!}\right)^{m_\beta}$ represents the sum of coefficients in the Bell polynomial and can be bounded by $C(n,D,k):=n!\left(1+n\binom{n+D-1}{D-1}\right)^k$.

Since the tensor Bell polynomial is a homogeneous polynomial of degree $k$, we can use similar estimation techniques to bound the gradient of the Bell polynomial with respect to its arguments. For arguments bounded by $B$, we have
\begin{equation*}
    \|\nabla B_{\alpha,k}(x)\| \le k\;C(\alpha,k)\,B^{\,k-1}.
\end{equation*}
Applying this to bound $\Delta B_{\alpha,k}(x)$ via the Mean Value Theorem, we obtain the bound for the second term:
\begin{equation*}
\begin{aligned}
        \left\|E_2(x)\right\|_{L_\infty}:&= \left\|\sum_{k=1}^{|\alpha|}D^k f\bigl(h(x)\bigr) \,\Delta B_{\alpha,k}(x)\right\|_{L_\infty}\\
        &\leq \sum_{k=1}^{|\alpha|}kC(\alpha,k)\left\|f\right\|_{W_\infty^{n}}\left\|g-h\right\|_{W^n_\infty}\left(\max\{\left\|g\right\|_{W^n_\infty},\left\|h\right\|_{W^n_\infty}\}\right)^{k-1}.
\end{aligned}
\end{equation*}
Combining these results, we arrive at the final inequality:
\begin{equation}\label{equation: composite difference}
    \begin{aligned}
        &\quad \left\|f(g(x))-f(h(x))\right\|_{W_\infty^n} \leq \sum_{k=1}^{n}C(n,D,k)\left\|f\right\|_{W_\infty^{n+1}}\left\|g-h\right\|_{L_\infty}\|g\|_{W^n_\infty}^k\\
        &+\sum_{k=1}^{n}kC(n,D,k)\left\|f\right\|_{W_\infty^{n}}\left\|g-h\right\|_{W^n_\infty}\left(\max\{\left\|g\right\|_{W^n_\infty},\left\|h\right\|_{W^n_\infty}\}\right)^{k-1}.
    \end{aligned}
\end{equation}
\subsection{Averaged Taylor Polynomials}\label{Subsection: Averaged Taylor Polynomials}
This subsection establishes a polynomial approximation result for functions in $W_p^k([0,1]^d)$. Since weak derivatives are not defined pointwise, we employ the concept of an averaged Taylor polynomial. For a function $f \in W_p^k(U)$ on a bounded open set $U \subset \mathbb{R}^d$, its averaged Taylor polynomial over a ball $B(x_0, r, \|\cdot\|_2)$ strictly contained within $U$ is defined as:
\begin{equation}\label{equation: averaged Taylor polynomial}
    Q^kf(x):=  \int_{B(x_0,r,\|\cdot\|_2)} \lambda(y) P_{y}^{k-1}f(x) dy.
\end{equation}
Here, $P_{y}^{k-1}f$ is the Taylor polynomial of degree $k-1$ for $f$ at point $y$, and $\lambda(y)$ is an arbitrary normalized cutoff function with support on $\overline{B(x_0,r,\|\cdot\|_2)}$.

With the averaged Taylor polynomial, we can use the Bramble-Hilbert Lemma to obtain a local approximation error bound. First, we introduce the necessary definitions.
\begin{definition}
    A set $U$ is said to be star-shaped with respect to a ball $B$ if for all $x \in U$, the closed convex hull of $\{x\}\cup B$ is a subset of $U$, i.e., $\overline{\mathrm{conv}}(\{x\}\cup B) \subset U$.
\end{definition}
\begin{definition}
    Let $\mathrm{diam}(U)$ be the diameter of $U$. Assume $U$ is star-shaped with respect to some ball. Let 
    \begin{equation*}
            \eta_{\max} = \sup\{\eta \ | \  U \text{ is star-shaped with respect to a ball of radius } \eta  \}.
    \end{equation*}
    The chunkiness parameter of $U$ is defined as
    \begin{equation*}
        \zeta = \frac{\mathrm{diam}(U)}{\eta_{\max}}.
    \end{equation*}
\end{definition}
We now state the Bramble-Hilbert Lemma.
\begin{lemma}[Bramble-Hilbert Lemma] \label{lemma: Bramble-Hilbert}
 Let $f \in W_p^k(U)$ be a function on a bounded, open, and star-shaped domain $U$. For any integer $s \in \{0, 1, \dots, k\}$ and for $p \geq 1$, its averaged Taylor polynomial $Q^kf$ of degree $k-1$ satisfies:
   \begin{equation}\label{equation: Bramble-Hilbert lemma}
       |f-Q^kf|_{W_p^s(U)} \leq C(k,d,\zeta) (\mathrm{diam}(U))^{k-s} |f|_{W_p^k(U)},
   \end{equation}
   where the constant $C(k,d,\zeta)$ depends only on $k$, $d$, and the chunkiness parameter $\zeta$.
\end{lemma}
A proof can be found in~\cite[Lemma 4.3.8]{brenner2008mathematical}.
Now, we apply this result by setting $U = T_{\,i_1,\dots,i_d}$. The diameter is $\mathrm{diam}(U) = \sqrt{\sum_{j=1}^d(t_{j,i_j+k}-t_{j,i_j})^2} \leq \sqrt{d}k/N$. Let $x_c$ be the center of the cell $T_{\,i_1,\dots,i_d}$, and let $l_{\min} = \min_{j=1,\ldots,d} |t_{j,i_j+k}-t_{j,i_j}|$. One can verify that for the cell $T_{\,i_1,\dots,i_d}$, we have $\eta_{\max} = l_{\min}/2$. Given our partition, $\eta_{\max} \in [1/(2N), k/(2N)]$, which implies that the chunkiness parameter $\zeta$ is bounded, i.e., $\zeta \in [2\sqrt{d},2k\sqrt{d}]$. We choose the ball $B=B(x_c, r, \|\cdot\|_2)$ with radius $r=(3/8)l_{\min}$. This choice ensures that the condition $r > (1/4)\eta_{\max}$ (a common requirement, check reference) is satisfied and $B \subset T_{\,i_1,\dots,i_d}$. Applying~\autoref{lemma: Bramble-Hilbert}, we conclude that there exists an averaged Taylor polynomial $Q^kf$ satisfying the following inequality:
\begin{equation*}
   |f-Q^kf|_{W_p^s(T_{\,i_1,\dots,i_d})} \leq C(k,d,\zeta) (\sqrt{d}k)^{k-s}N^{-(k-s)} |f|_{W_p^k(T_{\,i_1,\dots,i_d})} ,\quad  s=0,1,\ldots,k.
\end{equation*}
\section{Laplace--Beltrami Operator on Spheres and the Torus}\label{appendix: LB and hd}
This section collects basic formulae and spectral descriptions of the Laplace--Beltrami operator on the unit sphere and on an embedded torus.  See~\cite{dai2013approximation} for further details on spherical harmonics.
\subsection{Sphere Laplacian}
Let the unit sphere be \(\mathbb S^{d-1}=\{x\in\mathbb R^d:\|x\|=1\}\).  For any smooth \(f:\mathbb S^{d-1}\to\mathbb R\) take a smooth extension \(\tilde f\) to a neighborhood in \(\mathbb R^d\) (for instance the $0$-homogeneous extension \(\tilde f(x)=f(x/\|x\|)\)).  At \(x\in\mathbb S^{d-1}\) define the tangential projection matrix
\[
P_x := I - x x^\top \in\mathbb R^{d\times d},
\]
so that the surface (tangential) gradient can be written as
\[
\nabla_{\mathbb S^{d-1}} f(x)=P_x\,\nabla_{\mathbb R^d}\tilde f(x),\qquad |x|=1.
\]
The Laplace--Beltrami operator on the sphere is the tangential divergence of the tangential gradient:
\[
\Delta_{\mathbb S^{d-1}} f(x)
= \operatorname{div}_{\mathbb R^d}\!\big(P_x\,\nabla_{\mathbb R^d}\tilde f\big)(x),\qquad x\in\mathbb S^{d-1},
\]
and this representation is independent of the chosen extension \(\tilde f\).  Expanding in ambient coordinates yields the angular Laplacian
\[
\Delta_0
=\sum_{i=1}^d\frac{\partial^2}{\partial x_i^2}
-\sum_{i=1}^d\sum_{j=1}^d x_i x_j\frac{\partial^2}{\partial x_i\partial x_j}
-(d-1)\sum_{i=1}^d x_i\frac{\partial}{\partial x_i}.
\]
Introducing angular derivatives
\[
D_{i,j}=x_i\frac{\partial}{\partial x_j}-x_j\frac{\partial}{\partial x_i},\qquad 1\le i<j\le d,
\]
one has the angular representation
\[
\Delta_{\mathbb S^{d-1}} = \sum_{1\le i<j\le d} D_{i,j}^2.
\]
Spherical harmonics are defined spectrally as eigenfunctions of \(\Delta_{\mathbb S^{d-1}}\).  For each integer \(n\ge0\) the eigenvalue is
\[
\lambda_n = n(n+d-2),
\]
and the corresponding eigenspace is given by the restrictions to the sphere of homogeneous harmonic polynomials of degree \(n\); its multiplicity is
\[
m(n)=\binom{n+d-1}{n}-\binom{n+d-3}{n-2}.
\]
In \(L^2(\mathbb S^{d-1})\) choose an orthonormal basis \(\{Y_{n,k}\}_{n\ge0,\;1\le k\le m(n)}\) so that
\[
\Delta_{\mathbb S^{d-1}} Y_{n,k} = -\lambda_n\,Y_{n,k}.
\]
Any \(f\in L^2(\mathbb S^{d-1})\) admits the expansion
\[
f(\theta)=\sum_{n=0}^\infty\sum_{k=1}^{m(n)} f_{n,k}\,Y_{n,k}(\theta),\qquad
f_{n,k}=\int_{\mathbb S^{d-1}} f(\theta)\,\overline{Y_{n,k}(\theta)}\,\mathrm d\omega(\theta),
\]
with Parseval's identity
\[
\|f\|_{L^2(\mathbb S^{d-1})}^2=\sum_{n=0}^\infty\sum_{k=1}^{m(n)} |f_{n,k}|^2.
\]
From the spectral representation one obtains the energy identities
\[
\|\nabla_{\mathbb S^{d-1}} f\|_{L^2}^2=\sum_{n,k}\lambda_n\,|f_{n,k}|^2,\qquad
\|\Delta_{\mathbb S^{d-1}} f\|_{L^2}^2=\sum_{n,k}\lambda_n^2\,|f_{n,k}|^2,
\]
and hence the Sobolev \(H^s(\mathbb S^{d-1})\) norm admits the spectral weight representation
\[
\|f\|_{H^s(\mathbb S^{d-1})}^2 \simeq \sum_{n=0}^\infty (1+\lambda_n)^s\sum_{k=1}^{m(n)}|f_{n,k}|^2.
\]
\subsection{Torus Laplacian}
We consider the embedded torus
\[
\mathbb T^2_{R,r}=X_{R,r}([0,2\pi)^2)\subset\mathbb R^3,\qquad
X_{R,r}(u,v)=\big((R+r\cos v)\cos u,\; (R+r\cos v)\sin u,\; r\sin v\big),
\]
with \(R>r>0\).  The first fundamental form induced by this parametrization is diagonal,
\[
g=\begin{pmatrix}
(R+r\cos v)^2 & 0\\[4pt]
0 & r^2
\end{pmatrix},
\]
so the area element is \(\sqrt{|g|}=r\,(R+r\cos v)\).  The inverse metric components are \(g^{uu}=(R+r\cos v)^{-2}\), \(g^{vv}=r^{-2}\).  Therefore the Laplace--Beltrami operator in the coordinates \((u,v)\) takes the divergence form
\begin{equation}\label{eq:torus-lb-coord}
\Delta_{\mathbb T^2_{R,r}} f
=\frac{1}{\sqrt{|g|}}\Big[\partial_u\big(\sqrt{|g|}\,g^{uu}\,\partial_u f\big)
+\partial_v\big(\sqrt{|g|}\,g^{vv}\,\partial_v f\big)\Big],
\end{equation}
or equivalently
\[
\Delta_{\mathbb T^2_{R,r}} f
=\frac{1}{r(R+r\cos v)}\Big\{\partial_u\!\Big(\frac{r}{R+r\cos v}\partial_u f\Big)
+\partial_v\!\Big(\frac{R+r\cos v}{r}\partial_v f\Big)\Big\}.
\]

For the analytic Fourier-type family used in the experiments one obtains closed-form expressions.  If
\[
g(u,v)=A\cos(m u)\qquad(m\in\mathbb Z_{\ge0}),
\]
then \(g\) is independent of \(v\) and from \eqref{eq:torus-lb-coord} one gets
\begin{equation}\label{eq:lap-on-cosmu}
\Delta_{\mathbb T^2_{R,r}}[A\cos(m u)]
= -\frac{m^2\,A\cos(m u)}{(R+r\cos v)^2},
\end{equation}
which is explicit (the right-hand side depends on \(v\) through the metric factor).  For a \(v\)-mode
\[
h(u,v)=B\sin(n v)\qquad(n\in\mathbb Z_{\ge0}),
\]
a direct computation yields the closed form
\begin{equation}\label{eq:lap-on-sinnv}
\Delta_{\mathbb T^2_{R,r}}[B\sin(n v)]
=\frac{B\,n}{r^2(R+r\cos v)}\Big[-r\sin v\cos(n v)-n(R+r\cos v)\sin(n v)\Big],
\end{equation}
and analogous expressions hold for \(\cos(nv)\) and linear combinations thereof.  These closed-form images make \(\Delta_{\mathbb T^2_{R,r}}g\) available exactly for supervision in the loss.

The formulae above provide the pointwise and spectral tools used in the numerical experiments: for the sphere one uses the ambient-space/angular representations and spherical-harmonic expansion, while for the torus one evaluates \(\Delta_{\mathbb T^2_{R,r}}\) via \eqref{eq:torus-lb-coord} and uses \eqref{eq:lap-on-cosmu}--\eqref{eq:lap-on-sinnv} for the analytic Fourier targets.

\bibliographystyle{plain}
\bibliography{main.bib}
\end{document}